\documentclass[12pt]{amsart}
\usepackage{amscd,amssymb,graphics,color,a4wide,hyperref,verbatim}
\usepackage[enableskew]{youngtab}

\usepackage{multicol, fullpage} 
\usepackage[usenames,dvipsnames]{xcolor} 
\usepackage{tikz, tikz-3dplot, pgfplots}
\usepackage{tikz-cd}
\usetikzlibrary[positioning,patterns] 
\usetikzlibrary{matrix,arrows,decorations.pathmorphing}

\definecolor{light-gray}{gray}{0.7}

\usepackage{graphicx}
\usepackage{psfrag}
\usepackage{marvosym}
\usepackage{amsmath}
\usepackage{amsfonts}
\usepackage{amssymb}
\usepackage{amsthm}
\usepackage{mathrsfs}
\usepackage{enumitem}

\usepackage{ytableau}

\input xy
\xyoption{all}

\usepackage{calligra}
\usepackage{mathrsfs}
\DeclareMathAlphabet{\mathcalligra}{T1}{calligra}{m}{n}
\DeclareFontShape{T1}{calligra}{m}{n}{<->s*[1.5]callig15}{}

\newtheorem{theorem}{Theorem}[section]
\newtheorem*{theoremstar}{Theorem}
\newtheorem{lemma}[theorem]{Lemma}
\newtheorem{proposition}[theorem]{Proposition}
\newtheorem{corollary}[theorem]{Corollary}

\theoremstyle{definition}
\newtheorem{definition}[theorem]{Definition}

\newtheorem{example}[theorem]{Example}

\newtheorem{remark}[theorem]{Remark}
\newtheorem{theorem-definition}[theorem]{Theorem-Definition}

\numberwithin{equation}{section}

\renewcommand{\AA} {\mathbb{A}}

\newcommand{\CC} {\mathbb{C}}
\newcommand{\DD} {\mathbb{D}}

\newcommand{\LL} {\mathbb{L}}

\newcommand{\PP} {\mathbb{P}}
\newcommand{\QQ} {\mathbb{Q}}
\newcommand{\RR} {\mathbb{R}}

\newcommand{\ZZ} {\mathbb{Z}}

\newcommand {\shH} {\mathcal{H}}

\newcommand {\shQ} {\mathcal{Q}}

\newcommand {\shS} {\mathcal{S}}
\newcommand {\shT} {\mathcal{T}}

\newcommand {\shU} {\mathcal{U}}

\newcommand {\shZ} {\mathcal{Z}}

\newcommand {\sE} {\mathscr{E}}
\newcommand {\sF} {\mathscr{F}}
\newcommand {\sG} {\mathscr{G}}

\newcommand {\sI} {\mathscr{I}}

\newcommand {\sK} {\mathscr{K}}
\newcommand {\sL} {\mathscr{L}}

\newcommand {\sN} {\mathscr{N}}
\newcommand {\sO} {\mathscr{O}}
\newcommand {\sP} {\mathscr{P}}
\newcommand {\sQ} {\mathscr{Q}}

\newcommand {\sV} {\mathscr{V}}
\newcommand {\sW} {\mathscr{W}}

\newcommand {\foh}  {\mathfrak{h}}

\newcommand{\blank}{\underline{\hphantom{A}}}

\newcommand {\codim} {\operatorname{codim}}
\newcommand {\Coh} {\operatorname{Coh}}

\newcommand {\Coker} {\operatorname{Coker}}

\newcommand{\sExt}{\mathscr{E} \kern -1pt xt}

\newcommand {\Gr} {\operatorname{Gr}}

\newcommand{\Hilb}{\mathrm{Hilb}}

\newcommand {\Hom} {\operatorname{Hom}}
\newcommand {\sHom}{\mathscr{H}\kern-5pt\mathcalligra{om}}

\newcommand {\Id} {\operatorname{Id}}
\newcommand {\im} {\operatorname{im}}
\renewcommand {\Im} {\operatorname{Im}}

\newcommand {\Jac} {\operatorname{Jac}}

\newcommand {\kk} {\Bbbk}
\renewcommand {\ker } {\operatorname{Ker}}
\newcommand {\Ker} {\operatorname{Ker}}

\newcommand {\Pic} {\operatorname{Pic}}

\newcommand {\Proj} {\operatorname{Proj}}

\newcommand {\pr} {\operatorname{pr}}

\newcommand {\rank} {\operatorname{rank}}

\newcommand {\Span} {\operatorname{Span}}
\newcommand {\Spec} {\operatorname{Spec}}

\newcommand {\supp} {\operatorname{supp}}
\newcommand {\Supp} {\operatorname{Supp}}
\newcommand {\Sym} {\operatorname{Sym}}

\newcommand{\sTor}{\mathscr{T} \kern -3pt or}

\newcommand {\Bl} {\operatorname{Bl}}
\newcommand {\Quot} {\operatorname{Quot}}
\newcommand {\foQuot} {\mathfrak{Quot}}

\title[]{On the Chow theory of Quot schemes of locally free quotients}
\author[Q.Y.\ JIANG]{Qingyuan Jiang}

\address{School of Mathematics, University of Edinburgh, James Clerk Maxwell Building, Peter Guthrie Tait Road, Edinburgh EH9 3FD, United Kingdom.}
\email{qingyuan.jiang@ed.ac.uk}

\begin{document}

\begin{abstract} We prove a formula for Chow groups of $\Quot$-schemes which resolve degeneracy loci of a map between vector bundles, under expected dimension conditions. This result provides a unified way to understand the formulae for various geometric situations such as blowups, Cayley's trick, projectivizations, Grassmannian bundles, as well as Gassmannian type flips/flops and virtual flips. We also give applications to blowups of determinantal ideals, moduli spaces of linear series on curves, and Hilbert schemes of points on surfaces.
\vspace{-2mm} 
\end{abstract}

\maketitle

\vspace{-1 em} 
\section{Introduction}

For a Cohen-Macaulay scheme $X$ over a field $\kk$ of characteristic zero, let 
 $\sG$ be a coherent sheaf on $X$ which has homological dimension $\le 1$. (If $X$ is regular, then this condition is equivalent to $\sExt^i_{\sO_X}(\sG,\sO_X) =0, \forall i \ge 2$.) For any integer $d\ge0$, consider the $\Quot$-scheme $\foQuot_{X,d}(\sG) = \foQuot_{\sG/X/X}^{d, \sO_X}$ of rank $d$ locally free quotients of $\sG$, i.e. for every $T \to X$,
 	$$\foQuot_{X,d}(\sG)(T):= \{(\sE, q) \mid \text{$\sE$ locally free of rank $d$ on $T$}, q \colon \sG_T \twoheadrightarrow \sE ~\text{is $\sO_{T}$-linear}\} / \sim$$
see \S \ref{sec:Quot} and \cite{Gro,Nit}. By convention $\foQuot_{X, 0}(\sG)=X$, and $\foQuot_{X,d} (\sG)= \emptyset$ if $d<0$. Let
	$$\sK : = \sExt^1_{\sO_X}(\sG,\sO_X), \qquad \delta: = \rank \sG.$$
It is shown in the noncommutative counterpart of this paper \cite{J20} that under expected dimension conditions, the following relation holds in the Grothendieck ring 
 $K_0(Var_{\,\kk})$: 
	\begin{align*} [ \foQuot_{X,d}(\sG) ] = \sum_{j=0}^{\min\{d, \,\delta\}} \LL^{(d-j)(\delta-j)+\ell} \cdot [\foQuot_{X,d-j}(\sK)]^{\oplus b_{\ell}^{(j,\delta)}} \in K_0(Var_{\,\kk}),
	\end{align*} 
where $\LL = [\AA^1]$, and $b^{(d,k)}_{i} := b_{2i}(\Gr_d(k))$ is the $2i$-th Betti number of the Grassmannian $\Gr_d(k)$. In this paper we verify the relations on the level of Chow groups and motifs. 
 \begin{theoremstar}[Quot--formula, See Thm. \ref{thm:main} \& Cor. \ref{cor:main.motif}] If the degeneracy loci of $\sG$ have expected dimensions, then there is an decomposition of integral Chow groups: $\forall k \in \ZZ$,
	$$ CH^k(\foQuot_{X,d}(\sG) ) \simeq \bigoplus_{j=0}^{\min\{d,\delta\}} \bigoplus_{\ell=0}^{j(\delta-j)} CH^{k - (d-j)(\delta - j) - \ell} (\foQuot_{X,d-j}(\sK))^{\oplus b_{\ell}^{(j, \delta)}},$$
and under regular conditions a decomposition of contravariant Chow motives:
	$$h(\foQuot_{X,d}(\sG)) = \bigoplus_{j=0}^{\min\{d,\delta\}} \bigoplus_{\ell=0}^{j(\delta-j)} \left(h(\foQuot_{X,d-j}(\sK)) \otimes L^{(d-j)(\delta - j) + \ell}\right)^{\oplus b_{\ell}^{(j,\delta)}},$$
where $L = 1(-1)$ is the (contravariant) Lefschetz motives. 
 \end{theoremstar}
 
 There are in general two types of behaviours of $\foQuot_{X,d}(\sG)$ depending on the value of $d$:
	\begin{itemize}
 		\item If $d \le \rank \sG \equiv \delta$, then $\pi \colon \foQuot_{X,d}(\sG) \to X$ is generically a Grassmannian bundle of fiber $\Gr_{d}(\delta)$. The theorem implies there is a part of $CH(\foQuot_{X,d}(\sG))$ given by the same formula as a Grassmannian $\Gr_{d}(\delta)$-bundle over $X$ of \S \ref{sec:Gr} , with complementary summands given by Chow groups of resolutions of degeneracy loci \S \ref{sec:deg} of $\sG$.
		\item If $d > \rank \sG \equiv \delta$, then $\foQuot_{X,d}(\sG)$ and $\foQuot_{X,d-\delta}(\sK)$ both maps birationally to the degeneracy locus $X^{\ge d}(\sG) = \{x \mid \rank \sG(x) \ge d\} \subseteq X$, see \S \ref{sec:deg}.
		\begin{enumerate}
			\item[(i)] If $\delta=0$, then $\foQuot_{X,d}(\sG) \dashrightarrow \foQuot_{X,d-\delta}(\sK)$ is a flop (arising from two different Springer type desingularizations of the degeneracy loci), and the theorem implies there is an isomorphism $CH(\foQuot_{X,d}(\sK)) \simeq CH(\foQuot_{X,d}(\sG))$. 
			\item[(ii)] If $\delta>0$, then $\foQuot_{X,d}(\sG) \dashrightarrow \foQuot_{X,d-\delta}(\sK)$ is a flip. The theorem implies there is an embedding $CH(\foQuot_{X,d-\delta}(\sK)) \hookrightarrow CH(\foQuot_{X,d}(\sG))$, with complementary summands explicitly given by resolutions of higher degeneracy loci.
		\end{enumerate}
   	\end{itemize}
	
This ``Quot--formula" also provides a uniformed way to understand different formulae:
	\begin{enumerate}
 		\item If $\sG$ is locally free, then $\foQuot_{X,0}(\sK) = X$ and $\foQuot_{X,d-j}(\sK)=\emptyset$ for $j<d$. The theorem reduces to the well-know formula for {\em Grassmannian bundles} \S \ref{sec:Gr}; 
		\item If $\sG = \Coker (\sO_X \xrightarrow{s} E)$ for a regular section $s \in H^0(X,E)$ of a vector bundle $E$, then $\foQuot_{X,0}(\sK) = X$, $\foQuot_{X,1}(\sK)= Z := {\rm Zeros}(s) \subset X$ and $\foQuot_{X,d-j}(\sK)=\emptyset$ for $j<d-1$. The theorem becomes a formula for {\em generalised Cayley's trick}:
			$$ CH^k(\foQuot_{X,d}(\sG) ) \simeq \bigoplus_{\ell=0}^{(d-1)(\delta-d+1)} (CH^{k-(\delta-d+1)-\ell}(Z))^{\oplus b_{\ell}^{(d-1,\delta)}} \oplus  \bigoplus_{\ell=0}^{d (\delta-d)} (CH^{k-\ell}(X))^{\oplus b_{\ell}^{(d,\delta-1)}},$$
		 see Thm. \ref{thm:Cayley} in \S \ref{sec:Cayley} for more details. This itself generalizes both the {\em blowup formula} ($d=\delta$) and {\em Cayley's trick} in \cite{J19} ($d =1$), see Example \ref{ex:cayley.blowup}.
		\item if $d=1$, this becomes the {\em projectivization formula} proved in \cite{J19}
		 which itself has many applications such as to symmetric powers of curves, nested Hilbert schemes, and the situation of Voisin maps for cubic fourfolds, see \cite{J19}. Notice our assumption corresponds to condition (B) in \cite{J19}.
 	\end{enumerate}
   
\begin{remark}
We call this ``Quot--formula", as it is a sequel and a generalization of the ``projectivization formula" of \cite{J19, JL18}. This verifies a conjecture in \cite{J20}, where the behaviour of the derived categories is studied. 
 \end{remark}

 \subsection{Applications to blowup of determinantal subscheme} \label{sec:intro:blowup} If $d = \delta$, then $\foQuot_d(\sG) = \Bl_Z X$ is the blowup of $X$ along the determinantal subscheme $Z = X^{\ge \delta+1}(\sG)$, see Lem. \ref{lem:Quot=Bl}, where $Z$ is Cohen--Macaulay codimension $\delta+1$. There is a stratification $\cdots \subset Z_{1} \subset Z_{0} = Z$ of $Z$ by the rank of $\sG$, and $\foQuot_{X, i+1}(\sK)=: \widetilde{Z_i} $ is a IH-small resolution of $Z_i$, for $i \ge 0$.
 
\begin{theoremstar}[See Thm. \ref{thm:Bl:det}] Let $X$ be an irreducible scheme, $Z \subset X$ be a determinantal subscheme of codimension $\delta+1$ whose strata satisfy expected dimension conditions as the Quot-formula. Then for any $k \ge 0$, there is an isomorphism of Chow groups:
	\begin{align*}
	CH^k(\Bl_{Z} X) \simeq CH^k(X) \oplus \bigoplus_{\ell =0}^{\delta-1} CH^{k-1-\ell}(\widetilde{Z}) \oplus \bigoplus_{i =2}^{\delta} \bigoplus_{\ell=0}^{i(\delta-i)} CH^{k- i^2-\ell}(\widetilde{Z_{i-1}})^{\oplus b_{\ell}(i,\delta)}.
	\end{align*}
A similar decomposition holds for Chow motives if the schemes are smooth projective.
\end{theoremstar}
In particular, if the schemes are smooth projective varieties over $\CC$, then via Betti realisation there is an isomorphism of rational Hodge structure for any $k \in \ZZ$:
	\begin{align*}
	H^k(\Bl_{Z} X;\QQ) \simeq H^k(X,\QQ) \oplus \bigoplus_{\ell =0}^{\delta-1} IH^{k-2-2\ell}(Z,\QQ) \oplus \bigoplus_{i =2}^{\delta} \bigoplus_{\ell=0}^{i(\delta-i)} IH^{k- 2i^2-2\ell}(Z_{i-1},\QQ)^{\oplus b_{\ell}(i,\delta)}, 
	\end{align*}
where $IH$ denotes the intersection cohomology.	

The theorem is a generalisation of the usual blowup formula along locally complete intersection subschems Ex. \ref{ex:cayley.blowup}, and along Cohen--Macaulay codimension $2$ subscheme \cite{J19}, see Example \ref{ex:blowup:det}. The theorem shows how (resolutions of) degeneracy loci $Z_i, i\ge1$ of the centre $Z=Z_0$ contributes to the Chow group/cohomology of the blowup, if the centre $Z$ is singular determinantal subscheme of codimension $\ge 3$. Similar phenomenon  also occurs for derived categories 
\cite{J20}.

 \subsection{Applications to linear series on curves}
Let $C$ be a complex smooth projective curve of genus $g \ge 1$, and $G_k^{r}(C)= \{\text{$g_k^r$'s  on $C$}\}$ be the scheme parametrizing linear series  of degree $k$ and dimension $r$ on $C$, see \cite{ACGH}. By convention $G_k^{-1}: = \Pic^k(C)$). Then:

\begin{theoremstar}[See Thm. \ref{thm:curves}] If $C$ is a general complex smooth projective curve of genus $g \ge 1$, then for any $n \ge 0$, $r \ge 0$, there is an isomorphism of Chow groups:
	$$ CH^k(G_{g-1+n}^r(C)) \simeq \bigoplus_{j=0}^{\min\{n,r+1\}} \bigoplus_{\ell=0}^{j(n-j)} CH^{k - (r+1-j)(n - j) - \ell} (G_{g-1-n}^{r-j}(C))^{\oplus b_{\ell}^{(j, n)}},$$
and an isomorphism of contravariant integral Chow motives:
	$$
	 h(G_{g-1+n}^r(C)) \simeq \bigoplus_{j=0}^{\min\{n,r+1\}} \bigoplus_{\ell=0}^{j(n-j)} \big(h(G_{g-1-n}^{r-j}(C)) \otimes L^{(r+1-j)(n - j) + \ell)} \big)^{\oplus b_{\ell}^{(j, n)}}.
	$$
\end{theoremstar}
This generalises the formula of symmetric powers in \cite{J19}, which is the case  $r=0$:
	\begin{align*}
	 CH_{k}(C^{(g-1+n)}) \simeq CH_{k-n}(C^{(g-1-n)}) \oplus \bigoplus_{i=0}^{n-1} CH_{k-(n-1)+i}(\Jac(C)).
	\end{align*}
For specific $r$ or $g$, the requirement of $C$ being general could be relaxed, for example, the formula for symmetric powers above holds for {\em any} curve \cite{J19}. However, in general the schemes $G_{d}^r(C)$ may not have expected dimensions and may not be reduced or irreducible.

 \subsection{Applications to (nested) Hilbert schemes of point on surface.} Let $S$ be a smooth complex surface, for any $n \ge 0$, denote $\Hilb_n$ the Hilbert scheme of ideals of $S$ of colength $n$. For any $d \ge 1$, consider the generlised nested Hilbert scheme:
 	$$\Hilb_{n,n+d}^{\dagger}(S) : = \{ (I_n \supset I_{n+d}) \mid I_{n}/I_{n+d} \simeq \kappa(p)^{\oplus d} \text{~ for some $p \in S$}\} \subset \Hilb_n \times \Hilb_{n+d}.$$
By convention, we set $\Hilb_{n,n}^{\dagger}(S) = \Hilb_n \times S$, and $\Hilb_{n,n+d}^{\dagger}(S) = \emptyset$ if $d<0$. Notice that if $d=1$, $\Hilb_{n,n+1}^{\dagger}(S) = \Hilb_{n,n+1}(S)$ is the usual nested Hilbert scheme. 
 
 \begin{theoremstar}[See Thm. \ref{thm:Hilb}] For any $n,d \ge 1$ and any $k \ge 1$,  there is an isomorphism:
 	$$CH^k(\Hilb_{n,n+d}^{\dagger}(S)) \simeq CH^{k-d}(\Hilb_{n-d,n}^{\dagger}(S)) \oplus CH^{k}(\Hilb_{n-d+1,n}^{\dagger}(S)).$$
 \end{theoremstar}
 
 Notice if $d=1$ this recovers the formula for usual nested Hilbert scheme \cite{J19}. 
 
 \subsection{Applications to Brill--Noether theory of moduli space of K3 categories} Another fruitful source of examples comes from the Brill--Noether theory of Bridgeland moduli space of stable objects on K3 surfaces \cite{Markman, AT}, and more generally in K3 categories \cite{B,BCJ2}. The results of this paper apply verbatim to the correspondence spaces in these situations where the Brill--Noether loci are shown to be of expected dimensions.
  
\subsection*{Convention} In the introduction we use {\bf cotravariant} conventions to compare with the computations in Grothendieck group, however in the main body of the paper we will use homological indices and the {\bf covariant} convention of \cite{Ful}.
We fix a field $\kk$ of characteristic zero, and all schemes and morphisms are defined over $\kk$. Throughout this paper $X$ will be a {\bf reduced} locally Noetherian scheme of pure dimension, and $\sG$ be a coherent sheaf over $X$. For any set $\shS$, $\alpha,\beta \in \shS$, we use $\delta_{\alpha,\beta}$ to denote the Kronecker  delta function, i.e. $\delta_{\alpha,\beta} = 0$ if $\alpha \ne \beta$ and $\delta_{\alpha,\alpha} = 1 \in \ZZ$.
For motives we use the {\bf covariant} convention of \cite{Ful}, which is compatible with homological indices.  For a smooth scheme $X$ over some ground field $\kk$, denote by $\foh(X)$ its class $(X, \Id_X,0)$ in the Grothendieck's category of covariant Chow motives over $\kk$ (compared with $h$ in the introduction for contravariant motives). We use $\foh(X)(i)$ to denote the Tate twist $\foh(X) \otimes L^{i}$, where $L = (\PP^1, p =[\PP^1 \times \{0\}] )$ is the Tate motif. Notice that under the {\em covariant} convention, a morphism $f \colon X \to Y$ induces $f_* \colon \foh(X) \to \foh(Y)$ and $f^* \colon \foh(X) \to \foh(Y) (\dim Y - \dim X)$. Hence $\foh(\PP^1) = 1 \oplus L = 1 \oplus 1(1)$, where $1 = \foh(\Spec k)$. Also $CH_{k}(\foh(X)(n)) = CH_{k-n}(X)$. 

\subsection*{Acknowledgement} 
The author would like to thank Arend Bayer for many helpful discussions throughout this work, and thank Dougal Davis and Konstanze Rietsch for many helpful discussions on the combinatorial details of the main theorem. This work was supported by the Engineering and Physical Sciences Research Council [EP/R034826/1].

\section{Preliminaries}


\subsection{Quot-scheme} \label{sec:Quot} The Quot schemes introduced by Grothendieck \cite{Gro}, further developed by Mumford and by Altman--Kleiman, etc, plays an important role in modern algebraic geometry, especially for deformation theory and moduli problems. See Nitsure \cite{Nit} for a nice survey of the construction. 
We will be mainly concerned with Quot-schems of {\em Grassmannian type}, that is, for $\sE$ a coherent sheaf on $X$, consider the functor $\foQuot_{d}(\sE) := \foQuot_{\sE/X/X}^{d, \sO_X}$ which associates to any morphism $T \to X$ the set equivalence classes:
 	$$\foQuot_{X,d}(\sE)(T):= \{(\sP, q) \mid \text{$\sP \in \Coh(T)$ locally free of rank $d$}, q \colon \sE_T \twoheadrightarrow \sP ~\text{is $\sO_{T}$-linear}\} / \sim,$$
 where $\sE_T$ is the base-change of $\sE$ along $T \to X$, and two pairs are equivalent $(\sP,q) \sim (\sP',q')$ if  $\ker(q) = \ker(q')$. The functor is {\em representable} by a projective $X$-scheme, denoted by the same notation, together with a {\em tautological quotient bundle} $\sQ_d$ of rank $d$ and a quotient map $\pi^* \sE \twoheadrightarrow \sQ_d$, where $\pi \colon \foQuot_{X,d}(\sE) \to X$ is the natural projection map.
 
 \begin{example} \label{ex:proj}  The {\em projectivization} of $\sE$, denoted by $\PP(\sE) = \PP_X(\sE) :=\Proj \Sym_{\sO_X}^\bullet \sE$, represents the Quot-scheme of rank $d=1$ locally free quotients: $\PP(\sE) = \foQuot_{X,1}(\sE)$. Therefore for any $X$-scheme $f \colon T \to X$, to give a $X$-morphism $\phi \colon T \to \PP_X(\sE)$ is equivalent to give a line bundle $\sL$ over $T$ together with a surjective $\sO_T$-module map $f^* \sE \twoheadrightarrow \sL$. If $\sE$ is locally free, we will also use notation $\PP_{\rm sub}(\sE) : = \PP(\sE^\vee)$.
\end{example}

\begin{example} If $\sE$ is locally free of rank $r$ over $X$, then for any integer $1 \le d \le r$, the {\em rank $d$ Grassmannian bundle of $\sE$ over $X$} is $\Gr_d(\sE) : = \Gr(\sE^\vee, d) : =  \foQuot_{d}(\sE^\vee)$ where $\sE^\vee: = \sHom_{\sO_X}(\sE,\sO_X).$ Therefore $\Gr_d(\sE)$ parametrizes rank $d$ sub-bundles of $\sE$, or equivalently rank $d$ locally free quotients of $\sE^\vee$. If $X = \Spec \kk$ and $\sE = V$ is a $\kk$-vector space of rank $r$, then we will simply call $\Gr_d(V) = \Gr_d(r)$ the {\em Grassmannian}. 
\end{example}

\subsection{Degeneracy loci} \label{sec:deg} Standard references are \cite{FP, Ful, GKZ, GG, Laz04}.
\begin{definition}
\begin{enumerate}[leftmargin=*]
	\item  Let $\sG$ be a coherent sheaf of (generic) rank $r$ over a scheme $X$. Denote 
	$$X^{ \ge k}(\sG): = \{x \in X \mid \rank \sG(x) \ge k\}  \quad \text{for} \quad k \in \ZZ$$
the degeneracy locus of $\sG$. Notice that $X^{ \ge k}(\sG) = X$ if $k \le r$, so by convention the {\em first degeneracy locus} or the {\em singular locus} of $\sG$ is defined to be $X_{\mathrm{sg}}(\sG) : = X^{\ge r+1}(\sG)$. 
	\item Let $\sigma: \sF \to \sE$ a map of $\sO_X$ modules between locally free sheaves $\sF$ and $\sE$ on $X$. The {\em degeneracy locus of $\sigma$ of rank $\ell$} is:
	$$D_\ell(\sigma) := \{ x \in X \mid \rank \sigma (x) \le \ell \}.$$
\end{enumerate}
\end{definition}
The degeneracy loci $X^{\ge k}(\sG)$ and $D_\ell(\sigma)$ are closed {\em subschemes} of $X$, with ideals generated by minors of the map $\sigma$, see for example \cite[\S 7,2]{Laz04}. The two notions are related as follows: for $\sigma \colon \sF \to \sE$ and $\sG: = \Coker \sigma$ to be the cokernel, then $X^{\ge k}(\sG) = D_{\rank \sE - k}(\sigma)$. 

The expected codimension of $D_{\ell}(\sigma) \subset X$ is $(\rank \sE -\ell)(\rank \sF -\ell)$, and if $\sG$ has homological dimension $\le 1$ (for example if $\sG: = \Coker \sigma$ such that $\delta := \rank \sG = \rank \sE - \rank \sF$) then the expected codimension of $X^{\ge \delta+i}(\sG) \subset X$ is $i(\delta+i)$ for $i \ge 0$.

If we consider the universal case $X=\Hom_\kk(W,V)$, the total space of maps between vector spaces $W$ and $V$ over a field $\kk$, and consider the tautological map $\sigma(A) = A$ for $A \in \Hom(W,V)$. Then the next lemma can be found or easily deduced from \cite{FP, GKZ, GG}. 

\begin{lemma} \label{lem:deg:universal} Let $A \in D_\ell \subset \Hom(W,V)$ to be a regular point of $D_\ell$, i.e. $A \in D_\ell \backslash D_{\ell-1}$. 
	\begin{enumerate}
		\item $T_{A} D_\ell = \{T \in \Hom(W,V) \mid T(\ker A) \subseteq \Im A\}$.
		\item $N_{D_\ell} X |_{A} = \Hom (\Ker A, \Coker A)$.
		\item $N_{D_\ell}^* X|_{A} = \{D \in \Hom(W,V) \mid DA = 0, AD = 0 \} = \Hom (\Coker A, \Ker A)$.
		\item $T_{A}^* D_\ell  
		= \Hom(W,V)/  \Hom (\Coker A, \Ker A)$.
	\end{enumerate}
\end{lemma}
More generally, if $\sigma: \sF \to \sE$ a map of vector bundles over $X$, and for a fixed $\ell$, consider the regular part of degeneracy loci $D: =  D_\ell(\sigma) \backslash D_{\ell-1}(\sigma)$. We have the following:

\begin{lemma} \label{lem:deg:normal} Assume $X$ is Cohen-Macaulay, and $D: =  D_\ell(\sigma) \backslash D_{\ell-1}(\sigma) \subset X$ has expected codimension $(\rank \sE -\ell)(\rank \sF -\ell)$. Note that by definition $\sigma|_D \colon \sF|_D \to \sF|_D$ has constant rank $\ell$ over $D$. Then $K: = \Ker \sigma|_D$ and $C: = \Coker \sigma|_D$ are locally free sheaves over $D$. Moreover, $D \subset X$ is a locally complete intersection subscheme, with  $N_{D/X} \simeq K^\vee \otimes C$.
\end{lemma}
\begin{proof} Consider $\pi \colon H = |\Hom_X(\sF,\sE)| \to X$ the total Hom space, and let $\DD_\ell \subset H$ be the universal degeneracy loci for the tautological map $\sigma' \colon \pi^* \sF \to \pi^* \sE$. Firstly, the desired results hold for $\DD: = \DD_\ell \backslash \DD_{\ell-1}$. In fact, by considering affine covers of $X$, we may assume $X = \Spec A$ and $\sF = W \otimes_\kk \sO_{\Spec A}$, $\sE = V \otimes_\kk \sO_{\Spec A}$ for vector spaces $W, V$. Then $H = \Hom(W, V) \times X$, $\DD_\ell = D_\ell \times X $, and the results follow from the point case Lem. \ref{lem:deg:universal}. Notice that $\Coker \sigma'$ and $\Ker \sigma'$ are locally free by Nakayama's lemma and our ongoing assumption that $X$ is reduced. 

Next, the map $\sigma \colon \sF \to \sE$ induces a section map $s_{\sigma} \colon X \to H$ such that $\sigma = s_{\sigma}^* \sigma'$, and $D = \DD \times_{X} H$ is the fiber product along the section. As the inclusion $\DD \hookrightarrow H$ is a locally complete intersection, $H$ and $X$ are Cohen-Macaulay, and the expected dimension condition implies ${\rm depth}(D,X) = \codim (D,X)$, hence the inclusion $D \hookrightarrow X$ is also locally complete intersection, with normal bundle $N_{D/X} = s_{\sigma}^* N_{\DD/H}$. Finally $s_{\sigma}^* N_{\DD/H} = K^\vee \otimes C$ since $K =  s_{\sigma}^* \Ker(\sigma')$ and $C = s_{\sigma}^* \Coker(\sigma')$, which follows from pulling back the exact sequence of vector bundles on $\DD$ to $D$.
\end{proof}

\subsection{Young diagram and Schur functions} \label{sec:Young} 
The standand references for this section are \cite{Ful, FulY, Mac}.
A partition $\lambda =  (\lambda_1, \lambda_2, \ldots, \lambda_d)$ of a positive integer $n$ is a sequence of integers whose sum $|\lambda|: = \sum_i \lambda_i = n$ and satisfy $\lambda_1 \ge \lambda_2 \ge \ldots \ge \lambda_d$. A partition $\lambda$ corresponds canonically to a Young diagram, also denoted by $\lambda$ by abuse of notations. There is a natural partial order $\subseteq$ of all Young diagrams given by inclusion; note that $\mu \subseteq \lambda$ if and only if $\mu_i  \le \lambda_i$ for all $i$. Denote for a non-negative integer $m$, denote $(m) : = (m,0,0,\ldots, 0)$ and $(m^d) = (m,m,\ldots,m)$. Denote $B_{d,\ell}$ the set of all Young diagrams subscribed in the rectangle of height $d$ and width $\ell$, denoted by $(d \times \ell)$, i.e $B_{d,\ell} = \{ \lambda \mid (0) \subseteq \lambda \subseteq (\ell^d) \} = \{\lambda = (\lambda_1, \lambda_2, \ldots, \lambda_d) \mid \ell \ge \lambda_1 \ge \ldots \ge \lambda_d \ge 0 \}$. For $\lambda \in B_{d,\ell}$, denote $\lambda^t \in B_{\ell,d}$ the {\em transpose} of the Young diagram $\lambda$, i.e. the diagram $\lambda^t$ obtained from $\lambda$ by reflection in the main diagonal; denote by $\lambda^c  = (\lambda^c_i :=\ell - \lambda_{d+1-i}) \in B_{d,\ell}$ the {\em complement} of $\lambda$ inside $B_{d , \ell}$. We will sometimes use the notation $\lambda^c  \equiv \lambda^{c_{d,\ell}} = (\ell^d) - \lambda = \ell - \lambda$ to indicate the dependence on $d,\ell$. 

For $\lambda \in B_{d_1,\ell_1}$ and $\nu \in B_{d_2,\ell_2}$, we denote by $\lambda \oslash \nu$ the standard Young diagram obtained by placing $\lambda$ below the rectangle $(d_2 \times \ell_1)$ and placing $\nu$ to the right of this rectangle. See the shaded region in the following diagram:
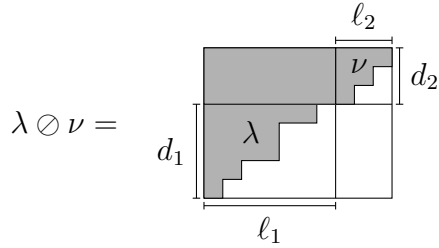
\begin{center}
\begin{tikzpicture}[scale=0.25]
    \draw (0,0) rectangle (10,8);
    \draw (0,5) -- (7,5) -- (7,0);
    \draw (7,8) -- (7,5) -- (10,5);
    
     \draw (-0.4,0) -- (-0.4,5);
    \draw (-0.6,0) -- (-0.2,0);
    \draw (-0.6,5) -- (-0.2,5);
    \node [left] at (-0.4,2.5) {$d_1$};
    
       \draw (0,-0.4) -- (7,-0.4);
    \draw (0,-0.2) -- (0,-0.6);
    \draw (7,-0.2) -- (7,-0.6);
    \node [below] at (3.5,-0.4) {$\ell_1$};
    
      \draw (7,8.4) -- (10,8.4);
    \draw (7,8.2) -- (7,8.6);
    \draw (10,8.2) -- (10,8.6);
    \node [above] at (8.5, 8.4) {$\ell_2$};
    
      \draw (10.4,5) -- (10.4,8);
    \draw (10.2,5) -- (10.6,5);
    \draw (10.2,8) -- (10.6,8);
    \node [right] at (10.4, 6.5) {$d_2$};
    
    \draw [fill=light-gray] (0,0) -- (0,5) -- (6,5) -- (6,4) -- (4,4) -- (4,2) -- (2,2) -- (2,1) -- (1,1) -- (1,0) -- (0,0);
    \node at (2.5,3.5) {$\lambda$};
    
       \draw [fill=light-gray] (7,5) -- (7,8) -- (10,8) -- (10,7) -- (9,7) -- (9,6) -- (8,6) -- (8,5) -- (7,5);
    \node at (8.25,7) {$\nu$};
    
     \draw [fill=light-gray] (0,5) -- (0,8) -- (7,8) -- (7,5) -- (0,5);
     
     \node[left] at (-4,4) {$\lambda \oslash \nu = $};
\end{tikzpicture}
\end{center}
We will simply use ${\scriptsize \young(*\nu,\lambda)}$ to indicate this diagram. Notice this operation depends on the two boxes, e.g. $(0) \oslash (0) = (\ell_1^{d_2})$. We hope its meaning is always clear from the context. 

It is also convenient to consider generalised partitions, i.e. a sequence of integers $\lambda = (\lambda_1 \ge \lambda_2 \ge \ldots \ge \lambda_d)$ whose entries $\lambda_i$ all allowed to be negative. There is an involution and a natural $\ZZ$-action on the set of all generalised Young diagrams of $d$ entries as follows. For $\lambda$, denote the involution of $\lambda$ by $-\lambda := (-\lambda_d, \ldots, -\lambda_1)$; for $k \in \ZZ$, there is an action $\lambda \mapsto \lambda+k :=  (\lambda_i + k) $. For two partitions $\lambda, \mu$, the sum is denoted by $\lambda + \mu : = (\lambda_i + \mu_i)$;  mote then the complement of $\lambda \in B_{d,\ell}$ is $\lambda^c = \ell - \lambda = (\ell^d) - \lambda$, justifying above notations.

For $\lambda \in B_{d,\ell}$ and a formal power series $c = \sum_i c_i t^i \in R[[t]]$ (with coefficients in some fixed commutative ring $R$ with unit $1 \in R$), denote the {\em Schur function} by
	$\Delta_{\lambda}(c): = \det (c_{\lambda_i+j-i})_{1 \le i,j \le d} \in R[t]$. 
For $\lambda, \mu \in B_{d,\ell}$, the {\em skew Schur function} is denoted by
	$\Delta_{\lambda/\mu} (c) : = \det (c_{\lambda_i - \mu_j +j-i})_{1 \le i,j \le d} \in R[t]$. {\em We will only consider the case $c_i = 0$ for $i<0$ and $c_0 = 1$}. The skew Schur function $\Delta_{\lambda/\mu} (c)$ is zero unless $\mu \subseteq \lambda$, in which case it depends only on the {\em skew Young diagram $\lambda/\mu$}, i.e. the set-theoretic difference of the diagrams $\lambda \backslash \mu$. Note that $\Delta_{\lambda}(c) = \Delta_{\lambda/(0)}(c)$; also $\Delta_{\lambda/\mu}(1) = 0$ unless $\lambda = \mu$, in which case $\Delta_{\lambda/\lambda} (1) = \Delta_{(0)}(1) = 1$.
If the formal power series $c = c(E)$ is given by the total Chern class of a vector bundle $E$ on a fixed scheme $X$, we will denote $\Delta_{\lambda/\mu}(E) := \Delta_{\lambda/\mu} (c(E))$. Note that the Schur function for $E$ only depends on its class in the K group $K_0(X)$, therefore the Schur function is naturally defined for the whole K group by $\Delta_{\lambda/\mu}(-E) := \Delta_{\lambda/\mu} (c(E)^{-1})$, and $\Delta_{\lambda/\mu}(E+F) := \Delta_{\lambda/\mu}(c(E) \cdot c(F))$.

For any partition $\lambda,\mu,\nu$ with $|\lambda| = |\mu| + |\nu|$, denote $c_{\mu, \nu}^{\lambda} \in \ZZ_{\ge 0}$ the {\em Littlewood-Richardson (LR) coefficients}, which counts the number of semi-standard Young tableaux of shape $\lambda /\mu$ and weight $\nu$ whose reverse word reading is a lattice permutation, see for example \cite{Ful, FulY, Mac}. Notice that $c_{\mu, \nu}^{\lambda} \ne 0$ implies $\mu \subseteq \lambda$ and $\nu \subseteq \lambda$. For any fixed $c  \in R[[t]]$ :
\begin{lemma}[Littlewood-Richardson rule] \label{lem:LR}
\begin{enumerate}
	\item  For all partitions $\mu, \nu$, the following holds:	
		$$\Delta_{\mu}(c) \cdot \Delta_{\nu}(c) = \sum_{\lambda} c_{\mu, \nu}^{\lambda} \Delta_\lambda (c).$$
	\item For any partitions $\mu \subseteq \lambda$, the following holds:
		$$\Delta_{\lambda/ \mu} (c) = \sum_{\nu \, \mid \,\nu \subseteq \lambda} c_{\mu,\nu}^\lambda \Delta_\nu (c).$$
\end{enumerate}
\end{lemma}

\begin{lemma}[Summation formula; see {\cite[Chapter I, (5.11)]{Mac}}] \label{lem:sum} Let $E_1, E_2, \ldots, E_m$ be $m$ vector bundles on a scheme $X$, and let $\lambda,\mu$ be partitions, then 
	$$\Delta_{\lambda / \mu} (E_1+E_2 + \ldots + E_m) = \sum_{(\nu)} \prod_{i=1}^m   \Delta_{\nu^{(i)}/\nu^{(i-1)}} (E_i) $$
where the summation runs through over all sequences $(\nu) = (\nu^{(0)},\nu^{(1)}, \ldots, \nu^{(m)})$of partitions, such that $\nu^{(0)} = \mu$, $\nu^{(m)} = \lambda$ and $\nu^{(0)} \subseteq \nu^{(1)} \subseteq \ldots \subseteq \nu^{(m)}$. In particular, 
	$$\Delta_{\lambda/ \mu}(E+F) = \sum_{\nu \, \mid \, \mu \subseteq \nu \subseteq \lambda} \Delta_{\lambda/ \nu}(E) \cdot \Delta_{\nu/ \mu}(F). $$
\end{lemma}

The {Littlewood-Richardson coefficients} $c_{\mu, \nu}^{\lambda}$ has many symmetries, described by {\em Berenstein--Zelevinsky (BZ)}'s triple multiplicities. For $\ell,m,n \in \ZZ_{\ge 0}^{d-1}$, Berenstein--Zelevinsky defined the a nonnegative integer $c_{\ell,m, n}$, {\em BZ number}, which counts the number of BZ patterns of type $(d-1;\ell,m,n)$, see \cite{BZ,Fom}. Let $\lambda, \mu, \nu$ are partitions with $d$-entries such that  $|\lambda| = |\mu| + |\nu|$, then BZ number is related to Littlewood-Richardson coefficient by:
	$$c_{\mu, \nu}^{\lambda} = c_{\ell(-\lambda), \ell(\mu), \ell(\nu)},$$
where for a generalize partition $\lambda=(\lambda_1,\ldots, \lambda_d)$, $\ell(\lambda): = (\lambda_1 - \lambda_2, \ldots, \lambda_{d-1} - \lambda_{d}) \in \ZZ_{\ge 0}^{d-1}$. For $\ell = (\ell_1, \ldots, \ell_{d-1}) \in  \ZZ_{\ge 0}^{d-1}$, denote $\ell^* : = (\ell_{d-1}, \ldots, \ell_1)  \in \ZZ_{\ge 0}^{d-1}$.

\begin{lemma}[Berenstein--Zelevinsky \cite{BZ}] \label{lem:BZ} The BZ numbers $c_{\ell, m, n}$ has a symmetry group of order $12$ generated by permutations of $\ell,m,n$ and the involution $(\ell,m,n) \mapsto (\ell^*, m^*, n^*)$, i.e.
	$$c_{\ell,m,n} = c_{m, \ell, n} = c_{\ell, n,m} = c_{n,m,\ell} = c_{m,n,\ell} = c_{n,\ell,m}, \qquad c_{\ell,m,n} = c_{\ell^*, m^*, n^*}.$$
\end{lemma}

It follows immediately from BZ-lemma that for any $\lambda, \mu, \nu \in B_{d,\ell}$ such that $|\lambda| = |\mu| + |\nu|$, for any $k \ge 0$, $c_{\mu, \nu}^{\lambda} =  c_{\mu+k, \nu}^{\lambda+k} = c_{\mu, \nu+k}^{\lambda+k}$, and $c_{\mu, \nu}^{\lambda} = c_{\nu, \mu}^{\lambda} = c_{\mu, \lambda^c}^{\nu^c} = c_{\mu, \nu^c}^{\lambda+\ell} = c_{m-\mu, n - \nu}^{m+n - \lambda}$ where $ \forall m, n \in \ZZ$ such that $m \ge |\mu|$, $n \ge |\nu|$. Therefore $\Delta_{\lambda/\mu} = \Delta_{\mu^c/\lambda^c} = \Delta_{(\lambda+k)/(\mu+k)}$. 

\begin{lemma}[See \cite{La, Ma} or {\cite[\S 4, Ex. 5, p.67]{Mac}}] \label{lem:tensor} If $E$ and $F$ are vector bundles over $X$ of rank $n$ and $m$ respectively, then the total Chern class of tensor product is
	$$c(E \otimes F) = \sum_{\mu \subseteq \lambda \subseteq (m)^n} d_{\lambda,\nu} \Delta_{\mu^t}(E) \Delta_{\lambda^c}(F) t^{nm - |\lambda| + |\mu|}$$
where $d_{\lambda,\nu} = \det \left[\binom{\lambda_i+n-i}{\mu_j + n - j} \right]_{1 \le i,j \le n}$. In particular, its top Chern class is:
	$$c_{\rm top}(E \otimes F) = c_{mn}(E \otimes F) = \sum_{\lambda \in B_{n,m}} \Delta_{\lambda^t} (E) \cdot \Delta_{\lambda^c}(F)$$
where $\lambda^c = m - \lambda$ denotes the complement inside $B_{n,m}$
	
\end{lemma}

\subsection{Grassmannian bundles}  \label{sec:Gr} The results of this subsection are classical, see for example \cite{Ful, Manin}. However for the purpose later usage of this paper, we present slight differently. Let $E$ be a vector bundle of rank $n$ over $X$, and let $d$ be an integer such that $1 \le d \le n$, denote by $\pi \colon \Gr_d(E) = \foQuot(E^\vee, d) \to X$ the Grassmannian bundle which parametrises rank $d$ subbundles of $E$ (or equivalently, rank $d$ locally free quotients of $E^\vee$). Denote $\shU = \shU(E)$ and $\shQ = \shQ(E)$ the universal subbundle resp. quotient bundle of rank $d$ resp. $\ell: = n-d$. Therefore there is a tautological short exact sequence 
	$$0 \to \shU \to \pi^* E \to \shQ \to 0.$$
The upshot is that the Chow group $CH(\Gr_d(E))$, regarded as a module over $CH(X)$, is freely generated by either one of the following basis, parametrised by a Young diagram $\lambda \in B_{d,\ell}$:
	\begin{align*}
	\Delta_{\lambda}  &: = \Delta_{\lambda}(-\shU) \equiv 
	 \det \big(c_{\lambda_i^t + j-i}(\shU^\vee)\big)_{1 \le i,j \le \ell} \in CH^{|\lambda|}(\Gr_d(E));  \\
	\Delta_{\lambda} ' &:= \Delta_{\lambda}(\shQ) \equiv \det \big(c_{\lambda_i + j-i}(\shQ)\big)_{1 \le i,j \le d}  \in CH^{|\lambda|}(\Gr_d(E)).
	\end{align*} 	
	
\begin{lemma}[Change of basis] \label{lem:Gr:change_basis} The basis $\{\Delta_{\lambda} : = \Delta_{\lambda}(-\shU) \}_{\lambda \in B_{d,\ell}}$ and $\{\Delta_{\lambda} ': = \Delta_{\lambda}(\shQ) \}_{\lambda \in B_{d,\ell}}$ are related by an invertible ``upper triangular" linear transformation:
	\begin{align*}
	\Delta_{\lambda}' = \sum_{\mu \,\mid\, (0) \subseteq \mu \subseteq \lambda}  \Delta_{\lambda/\mu}(E) \cdot \Delta_\mu, \qquad
	\Delta_{\lambda} = \sum_{\mu \,\mid\, (0) \subseteq \mu \subseteq \lambda}  \Delta_{\lambda/\mu}(-E) \cdot \Delta_\mu'.
	\end{align*}
In particular, for any fixed $\lambda \in B_{\ell,d}$, $\Span \{\Delta_{\mu} \}_{\mu \subseteq \lambda}= \Span \{\Delta_{\mu}' \}_{\mu \subseteq \lambda}$.
\end{lemma}
Here for a set of classes $\Delta_i \in CH^*(\Gr_d(E))$, $i \in I$, its $\Span$ denotes the subgroup $\Span\{\Delta_i\}_{i \in I} = \{\sum_i \Delta_i \cap \pi^* \alpha_i \mid \alpha_i \in CH(X)\} \subseteq CH(\Gr_d(E))$.

\begin{proof} In the $K_0(X)$ the following holds: $ -\shU = \shQ - \pi^* E$,  equivalently $ \shQ =  \pi^* E - \shU$. Therefore the results follows easily from summation formula Lem. \ref{lem:sum} in the case $\mu = (0)$.
\end{proof}

The following lemma is a variation of the well-know duality lemma, see also \cite[Thm. 2.5]{DP}. For the sake of completeness we present a proof in our setting. 
\begin{lemma}[Duality] \label{lem:Gr:dual} For any $\lambda, \mu \in B_{d,\ell}$, $\alpha \in CH(X)$,
	\begin{align*}
	\pi_* (\Delta_{\lambda^c} \cdot \Delta_{\mu} \cap \pi^* \alpha ) = \Delta_{\mu/\lambda}(-E) \cap \alpha, \qquad 
	\pi_* (\Delta_{\lambda^c}' \cdot \Delta_{\mu}' \cap  \pi^* \alpha ) = \Delta_{\mu/\lambda}(E) \cap \alpha.
	\end{align*}
In particular, above cycle classes are zero unless $\lambda \subseteq \mu$. Furthermore, the two basis $\{\Delta_\lambda\}$ and $\{\Delta_{\mu}'\}$ are dual to each other in the following sense:
	\begin{align*}
	\pi_* (\Delta_{\lambda^c} \cdot \Delta_{\mu}' \cap \pi^* \alpha ) = \delta_{\lambda,\mu} \cdot \alpha, \qquad 
	\pi_* (\Delta_{\lambda^c}' \cdot \Delta_{\mu} \cap  \pi^* \alpha ) = \delta_{\lambda,\mu}  \cdot \alpha.
	\end{align*}
\end{lemma}

\begin{proof} We only show the first identity; the second will follow whether by a similar computation, or from the first one by change of basis $\Delta_{\lambda}' = \sum_{\nu \subseteq \lambda} \Delta_{\lambda/\nu} (E) \cdot \Delta_{\nu}$:
	\begin{align*}
	&\pi_* (\Delta_{\lambda^c}' \cdot \Delta_{\mu}' \cap  \pi^* \alpha )  = \sum_{\nu^c \subseteq \lambda^c, \tau\subseteq \mu} \Delta_{\lambda^c/\nu^c}(E) \cdot  \Delta_{\mu/\tau}(E) \cap \pi_* (\Delta_{\nu^c} \cdot \Delta_{\tau} \cap \pi^* \alpha ) \\
	&= \sum_{\nu \supseteq \lambda, \tau \subseteq \mu}    \Delta_{\nu/\lambda}(E)  \cdot \Delta_{\mu/\tau}(E)  \cdot \Delta_{\tau/\nu}(-E) = \Delta_{\mu/\lambda}(E);
	\end{align*}
(Or one could observe that through identification $\Gr_d(E) \simeq \Gr_\ell(E^\vee)$, then $\shQ^\vee \simeq \shU(E^\vee)$,  $\Delta_{\lambda}(\shQ(E)) = \Delta_{\lambda^t}(-\shQ(E)^\vee) = \Delta_{\lambda^t}(-\shU(E^\vee))$, hence reduces to the first case.)

From degree reason we know $\pi_* \Delta_\lambda = 0$ for any $\lambda \in B_{d,\ell}$, unless $\lambda = (\ell^d)$ is the maximal element, in which case $\pi_* \Delta_{(\ell^d)} = 1$. For any partition $\lambda = (\lambda_1, \ldots, \lambda_d)$ with $\lambda_1>\ell$, by rank reason $\Delta_{\lambda}(\shQ) = 0$, since $\shQ = \pi^* E + (-\shU)$, by summation formula Lem. \ref{lem:sum}
	$$\Delta_{\lambda} + \sum_{\mu \, \mid \, \mu \subsetneq \lambda} \pi^*\Delta_{\lambda/\mu}(E)  \cdot \Delta_{\mu} = 0.$$
Therefore $\Delta_\lambda$ can be expressed by elements $\Delta_\mu$ with $\mu \subsetneq \lambda$. Hence by induction one obtains that $\pi_* \Delta_\lambda = 0$ for any $\lambda \nsupseteq (\ell^d)$. On the other hand, for any $\lambda \neq 0$, similarly $\Delta_{\lambda+\ell}(\shQ) = 0$ implies relations $\sum_{\mu} \pi^* \Delta_{\lambda /\mu} (E) \cdot \Delta_{\mu + \ell} = 0$. Pushing forward to $X$, one obtains
	$$ \sum_{\mu \, \mid \, \mu \subseteq \lambda} \Delta_{\lambda /\mu} (E) \cdot (\pi_* \Delta_{\mu + \ell}) = 0$$
 for all $\lambda \ne (0)$; note we already know $\pi_* \Delta_{(0)+
\ell} = 1$. Hence above relations uniquely determines the class $\pi_* \Delta_{\lambda + \ell}$ for all $\lambda$ by induction. On the other hand we know that $ \delta_{\lambda, (0)} = \Delta_\lambda (1) = \Delta_\lambda(E+ (-E)) =  \sum_{\mu} \Delta_{\lambda /\mu} (E) \Delta_{\mu}(-E)$, therefore comparing these relations one obtains that $\pi_* \Delta_{\lambda + \ell} = \Delta_{\lambda}(-E)$.

Back to the general case. For any $\lambda, \mu \in B_{d,\ell}$, by Littlewood--Richardson rule,
	\begin{align*}
	\Delta_{\lambda^c} \cdot \Delta_{\mu} = \sum_{\nu \in B_{\ell,d} } c_{\lambda^c, \, \mu}^{\nu + \ell} \Delta_{\nu + \ell} +  \sum_{\overline{\nu} \nsupseteq (\ell^d)} c_{\lambda^c, \, \mu}^{\overline{\nu} } \Delta_{\overline{\nu}}  =    \sum_{\nu \in B_{\ell,d} } c_{\lambda, \nu}^{\mu} \Delta_{\nu + \ell} +  \sum_{\overline{\nu} \nsupseteq (\ell^d)} c_{\lambda^c, \, \mu}^{\overline{\nu} } \Delta_{\overline{\nu}}.
	\end{align*}
Only the first term survives under $\pi_*$, hence by above calculation and projection formula,
	$$\pi_*(\Delta_{\lambda^c} \cdot \Delta_{\mu} \cap \pi^* \alpha) = \sum_\nu c_{\lambda,\nu}^\mu \Delta_\nu (-E) \cap \alpha  = \Delta_{\mu/\lambda} (-E) \cap \alpha.$$
The duality statement follows from the above equality by change of basis Lem. \ref{lem:Gr:change_basis}:
	$$\pi_* (\Delta_{\lambda^c} \cdot \Delta_{\mu}' \cap \pi^* \alpha) = (\sum_{\tau \subseteq \mu}  \Delta_{\mu/\tau}(E)  \Delta_{\tau/\lambda}(-E))  \cap \alpha = \delta_{\lambda,\mu} \alpha,$$
and similarly for the other equality.
\end{proof}

\begin{lemma}[Projectors] \label{lem:Gr:projector} For any $\lambda \in B_{d,\ell}$, if we denote 
	$$\pi_\lambda^*(\blank) := \Delta_{\lambda} \cap \pi^*(\blank)  \quad \text{and} \quad \pi_\lambda'^*(\blank) = \Delta_{\lambda}' \cap \pi^*(\blank),$$
 (both has homological degree $d\ell - |\lambda|$) and furthermore define for any $k\in \ZZ$:
	\begin{align*}
	& \pi_{\lambda *} (\blank): =  \sum_{\mu \, \mid  \, \lambda \subseteq \mu \subseteq (\ell^d)} \Delta_{\mu/\lambda}(E) \cap  \pi_* (\Delta_{\mu^c} \cap (\blank) ) \colon 
	& CH_{k}(\Gr_{d}(E)) \to CH_{k- d \ell + |\lambda|} (X) \\
	&   \pi_{\lambda *}' (\blank) :=  \sum_{\mu \, \mid \, \lambda \subseteq \mu \subseteq (\ell^d)} \Delta_{\mu/\lambda}(-E) \cap  \pi_* (\Delta_{\mu^c}' \cap (\blank) )
	\colon & CH_{k}(\Gr_{d}(E)) \to CH_{k- d \ell  + |\lambda|} (X).
	\end{align*}	
Then the following holds:
	$$ \pi_{\lambda *} \, \pi_{\mu}^* = \delta_{\lambda,\mu} \, \Id_{CH(X)}, \qquad  \pi_{\lambda *}' \,  \pi_{\mu}'^* = \delta_{\lambda,\mu}  \, \Id_{CH(X)}.$$	
\end{lemma}
\begin{proof} For any $\lambda, \tau \in B_{d,\ell}$, by duality lemma,
	\begin{align*}
	& \pi_{\lambda *} ( \pi_{\tau}^* \alpha)  = \pi_{\lambda *} ( \Delta_\tau \cap \pi^* \alpha) = \sum_{\mu \, \mid  \, \lambda \subseteq \mu \subseteq (\ell^d)} \Delta_{\mu/\lambda}(E) \cdot  \pi_* (\Delta_{\mu^c} \cdot \Delta_\tau  \cap \pi^* \alpha) \\
	& =  \sum_{\mu^c \, \mid \, (0) \subseteq \mu^c \subseteq (\ell^d)} \Delta_{\lambda^c/\mu^c}(E) \cdot \Delta_{\nu^c/\tau^c}(-E) \cap \alpha = \Delta_{\lambda^c/\tau^c} (E + (-E)) \cap \alpha = \delta_{\lambda,\tau} \cdot \alpha. \end{align*}
As before, $\pi_{\lambda}'^* \, \pi_{\mu\,*}' = \delta_{\lambda,\mu}$ follows from a similar argument, or from $\Gr_d(E) \simeq \Gr_\ell(E^\vee)$. 
\end{proof}

\begin{remark} It follows easily from Lem. \ref{lem:Gr:change_basis} that the maps in two basis are related by:
	\begin{align*}
	&	\pi_\lambda'^* (\blank)= \sum_{\mu \,\mid\,  (0) \subseteq \mu \subseteq \lambda} \Delta_{\lambda/\mu}(E) \cdot \pi_\mu^* (\blank), 
	\qquad \pi_\lambda^*  (\blank)= \sum_{\mu \,\mid\, (0) \subseteq \mu \subseteq \lambda} \Delta_{\lambda/\mu}(-E) \cdot \pi_\mu^* (\blank); \\
	&	\pi_{\lambda\,*}' (\blank) = \sum_{\mu \, \mid  \, \lambda \subseteq \mu \subseteq (\ell^d)} \Delta_{\mu/\lambda}(-E) \cdot \pi_{\mu\,*} (\blank), 
	\qquad \pi_{\lambda\,*} (\blank) = \sum_{\mu \, \mid  \, \lambda \subseteq \mu \subseteq (\ell^d)} \Delta_{\mu/\lambda}(E) \cdot \pi_{\mu\,*}' (\blank).
	\end{align*}
Notice that the formula for $\pi_{\lambda\,*}' $ could also follow easily from the for $\pi_{\lambda\,*}$.
\end{remark}

The classical results for $\Gr_d(E) \to X$ in \cite{Ful, Manin} can be now formulated as follows:

\begin{theorem}[Grassmannian bundle formula]  \label{thm:Grassmannian} For any $k \in \ZZ$, there are isomorphisms:
	\begin{align*}
	\bigoplus_{\lambda \in B_{d,\ell}} \pi_\lambda^* = \bigoplus_{\lambda \in B_{d,\ell}}  \Delta_{\lambda} \cap \pi^*(\blank)   \colon \bigoplus_{\lambda \in B_{d,\ell}}  CH_{k-d\ell+|\lambda|}(X) \simeq CH_k(\Gr_d(E)); \\
	\bigoplus_{\lambda \in B_{d,\ell}} \pi_\lambda'^* = \bigoplus_{\lambda \in B_{d,\ell}} \Delta_{\lambda}' \cap \pi^*(\blank)  \colon \bigoplus_{\lambda \in B_{d,\ell}}  CH_{k-d\ell+|\lambda|}(X) \simeq CH_k(\Gr_d(E)).
	\end{align*}
The two basis are related by an invertible linear transformation explicitly given by Lem. \ref{lem:Gr:change_basis}, with projectors respectively given by Lem. \ref{lem:Gr:projector}. Therefore the following holds:
	$$\Id_{CH(\Gr_d(E))} = \sum_{\lambda \in B_{d,\ell}} \pi_{\lambda}^* \,  \pi_{\lambda *}  =\sum_{\lambda \in B_{d,\ell}} \pi_{\lambda}'^* \,  \pi_{\lambda *}'. $$
By Manin's principle, the same maps induce isomorphisms of Chow motives:
	$$ \bigoplus_{\lambda \in B_{d,\ell}} \pi_\lambda^* \colon \bigoplus_{\lambda \in B_{d,\ell}} \foh(X)(d\ell-|\lambda|) \simeq \foh(\Gr_d(E)), \qquad  \bigoplus_{\lambda \in B_{d,\ell}} \pi_\lambda'^* \colon \bigoplus_{\lambda \in B_{d,\ell}} \foh(X)(d\ell-|\lambda|) \simeq \foh(\Gr_d(E)).$$
\end{theorem}
\begin{proof} See Fulton's book \cite{Ful}. Notice the injectivity of $\bigoplus \pi_\lambda^*$ and $\bigoplus \pi_\lambda'^*$ follows directly from Lem. \ref{lem:Gr:projector}. To show  surjectivity, it suffices to consider the case when $X$ is irreducible and $E$ is trivial, in which case the two basis coincides $\Delta_\lambda = \Delta_\lambda'$, and projectors of Lem. \ref{lem:Gr:projector} both take the simplest form $\pi_{\lambda *} = \pi_{\lambda *}' = \pi_* (\Delta_{\lambda^c} \cap (\blank))$. The theorem in this case follows from the absolute case $X = \Spec \kk$, when $\Delta_{\lambda^c} \cdot \Delta_{\mu} = \delta_{\lambda, \mu}$ holds by duality Lem. \ref{lem:Gr:dual}, and that $\Gr_d(E)$ admits an affine stratification by Schubert cells $\{\Delta_\lambda \cap [\Gr_d(E)]\}$.  \end{proof}

\begin{example} If $d=1$, $P:= \Gr_1(E) = \PP_{\rm sub}(E)$ is the projective bundle, and all Young diagram $\lambda \in B_{1,n-1}$ is of the form $\lambda = (i)$, $i \in [0,n-1]$. Then the two basis are given by $\{ \Delta_{(i)} = \zeta^i : = c_1(\sO(1))^i \}_{i \in [0, n-1]}$ and $\{ \Delta_{(i)} ' = c_i (\shT_{P/X}(-1))\}_{i  \in [0,n-1]}$. Lem. \ref{lem:Gr:change_basis} is nothing but
	$$c_i (\shT_{P/X}(-1)) = \sum_{j=0}^i c_{i-j}(E) \cdot \zeta^j, \qquad \zeta^i =  \sum_{j=0}^i s_{i-j}(E) \cdot c_i (\shT_{P/X}(-1)),$$
where $s_k(E)$ is the Segre class of $E$. Duality Lem. \ref{lem:Gr:dual} says for any $i,j \ge 0$ and $\alpha \in CH(X)$,
	$$\pi_* (\zeta^{i} \cdot \zeta^{j} \cap \pi^* (\alpha)) = s_{j+i-n+1}(E) \cap \alpha, \quad \pi_* (c_{i}(\shT_{P/X}(-1)) \cdot c_{j} (\shT_{P/X}(-1)) \cap \pi^* (\alpha)) = c_{j+i-n+1} (E) \cap \alpha.$$ 
Therefore the theorem translates into projective bundle formulae in two different basis, i.e. both the maps $\bigoplus_{i=0}^{n-1} \pi_i^* =\zeta^i \cap \pi^*(\blank)$ and $\bigoplus_{i=0}^{n-1} \pi_i'^* =c_i(\shT_{P/X}(-1)) \cap \pi^*(\blank)$ induce an isomorphism $ \bigoplus_{i=0}^{n-1}  CH_{k-(n-1)+i}(X) \simeq CH_k(P)$,
with projectors of Lem. \ref{lem:Gr:projector} given by:
	\begin{align*} 
	\pi_{i *} = \sum_{j=i}^{n-1} c_{j-i}(E) \pi_* (\zeta^{n-1-j} \cdot (\blank)), \quad
	\pi_{i *}' = \sum_{j=i}^{n-1} s_{j-i}(E) \pi_* (c_{n-1-j}(\shT_{P/X}(-1)) \cdot (\blank)).
	\end{align*}	
\end{example}

\section{Generalised Cayley's trick and virtual flips} \label{sec:Cayley}

\subsection{Generalised Cayley's trick} \label{sec:Cayley} The treatment of this section follows closely the Cayley's trick case in \cite[\S 3.1]{J19}. Let $E$ be a vector bundle of rank $n$ on $X$, and $s \in \Gamma(X,E)$ be a regular section, $Z=Z(s) \subset X$ be the zero locus. Let $d$ be an integer $1 \le d \le n-1$ and denote $\ell : = n-d$.  Let $\sG : = \Coker (s \colon \sO_X \to E)$ be the cokernel, denote $\shH_s : = \foQuot_d(\sG)$ and there is  a natural inclusion $\iota \colon \shH_s \hookrightarrow \foQuot_d(E) = \Gr_d (E^\vee)$ induced by $E \twoheadrightarrow \sG$. Over $G: = \Gr_d(E^\vee)$, denote by $\shU = \shU(E^\vee)$ (resp. $\shQ = \shQ(E^\vee)$) the universal rank $d$ (resp. $\ell : = n-d$) subbundle resp. (quotient bundle), then there is a short exact sequence 
	$$0 \to \shU \to \pi^* E^\vee \to \shQ \to 0.$$
 Notice the restriction of $\shU^\vee$ to $\shH_s$ is the tautological rank $d$ quotient bundle of $\sG$, i.e. $\pi^* \sG \twoheadrightarrow \shU^\vee(\sG) = \iota^* \shU^\vee$, which will be simply denoted by $\shU^\vee$ by abuse of notations. The situation is summarised in the following diagram, with names of maps as indicated:
\begin{equation}\label{diagram:Cayley}
	\begin{tikzcd}[row sep= 2.6 em, column sep = 2.6 em]
	G_Z:=\Gr_d(E^\vee|_Z) \ar{d}[swap]{p} \ar[hook]{r}{j} & \shH_s : =\foQuot_d(\sG) \ar{d}{\pi} \ar[hook]{r}{\iota} & G:=\Gr_d(E^\vee) \ar{ld}[near start]{q} 
	\\
	Z \ar[hook]{r}{i}         & X  
	\end{tikzcd}	
\end{equation}

We first introduce the following notations similar to \S \ref{sec:Gr}: for any $\lambda \in B_{d,\ell}$, denote 
	\begin{align*}
	\Delta_{\lambda}  &: = \Delta_{\lambda}(-\shU) \in CH^{|\lambda|}(G), \qquad \Delta_{\lambda} ' := \Delta_{\lambda}(\shQ) \in CH^{|\lambda|}(G).
	\end{align*} 	
By abuse of notations, we use the same notations $\Delta_{\lambda}  = \Delta_{\lambda} (-\shU(G_Z))$ and $\Delta_{\lambda} '  = \Delta_{\lambda} (\shQ(G_Z))$ to denote the corresponding classes on $CH(G_Z)$. 
 Then define $p_\lambda^*(\blank) := \Delta_{\lambda} \cap p^*(\blank)$, $p_\lambda'^*(\blank) = \Delta_{\lambda}' \cap p^*(\blank),$ $q_\lambda^*(\blank) := \Delta_{\lambda} \cap q^*(\blank)$, $q_\lambda'^*(\blank) = \Delta_{\lambda}' \cap q^*(\blank)$ as in \S \ref{sec:Gr}, and define the projectors $p_{\lambda *}$, $p_{\lambda *}'$, $q_{\lambda *}$, $q_{\lambda *}'$ by the same formulae of Lem. \ref{lem:Gr:projector}. Second, for any $\lambda \in B_{d,\ell-1}$, denote 
	\begin{align*}
	& \pi_\lambda^*(\blank) := \Delta_{\lambda} (-\shU)  \cap \pi^*(\blank) \colon & CH_{k-d(\ell-1) +|\lambda|}(X) \to CH_{k}(\shH_s), \\
	& \pi_{\lambda *} (\blank): =  \sum_{\mu \in B_{d, \ell-1}} \Delta_{\mu/\lambda}(\sG^\vee) \cap  \pi_* (\Delta_{\mu^c} (-\shU) \cap (\blank)) \colon  & CH_{k}(\shH_s) \to CH_{k-d(\ell-1) +|\lambda|}(X).
	\end{align*}
(Here $\mu^c  = \ell - 1 - \mu$.) Then clearly all pullbacks commute: $\pi_\lambda^* = \iota^* q_\lambda^*$, $p_\lambda^* = j^* \pi_\lambda^*$. For the projectors, 
it follows from $\Delta_{\mu/\lambda}(\sG^\vee) = \iota^* \Delta_{\mu/\lambda}(E^\vee)$, $\Delta_{\ell - 1 -\mu} (-\shU(\sG)) = \iota^* \Delta_{\ell - 1 -\mu}$, therefore $\iota_* \Delta_{\ell - 1 -\mu} (-\shU(\sG))  =  c_{d}(\shU^\vee) \cap \Delta_{\ell - 1 -\mu} =  \Delta_{(1^d)}  \cap \Delta_{\ell - 1 -\mu}  = \Delta_{\ell -\mu} $, and the following holds
	$$\pi_{\lambda *} = q_{\lambda+1\, *} \iota_*, \qquad \pi_{\lambda *} j_* = i_* p_{\lambda+1\,*}, \qquad \forall \lambda \in B_{d,\ell-1}.$$
 Finally, for any $\lambda \in B_{d-1,\ell}$, $k \in \ZZ$ (notice that $c(\sG) = c(\sE)$)), denote:
		\begin{align*} 
		&\Gamma_\lambda^* : = j_* \, p_{\lambda}'^* = j_* (\Delta_{\lambda}(\shQ) \cap p^*(\blank))  \colon  & CH_{k-(d-1)\ell +|\lambda|}(Z)  \to CH_{k}(\shH_s), \\
		&\Gamma_{\lambda *} : = p_{(\lambda^t+1)^t *}' \, j^*  =  \sum_{\mu \in B_{d-1,\ell}} \Delta_{\mu/\lambda}(-\sG^\vee) \, p_* (\Delta_{\mu^c} (\shQ)\cap j^*(\blank)) \colon &  CH_{k}(\shH_s) \to CH_{k-(d-1)\ell +|\lambda|}(Z).
		\end{align*}
(Here $\mu^c = \ell - \mu$.) The main result of this subsection is the following, which simultaneously generalises blowup formula ($d = n-1$, $\ell = 1$) and Cayley's trick ($d=1$; see \cite{J19}):

\begin{theorem}[Generalized Cayley's trick] \label{thm:Cayley} In the above situation, for any $k \in \ZZ$:
\begin{enumerate}
	 \item There exists a split short exact sequence:
	\begin{align*}
		0 \to \bigoplus_{\lambda \in B_{d,\ell-1}} CH_{k-d(\ell-1)+|\lambda|}(Z) \xrightarrow{f}  \bigoplus_{\lambda \in B_{d,\ell-1}} CH_{k-d(\ell-1)+|\lambda|}(X) \oplus CH_{k}(G_Z) \xrightarrow{g} CH_k(\shH_s) \to 0,
	\end{align*}
where the maps $f$ and $g$ is given by 
	\begin{align*}
	& f  \colon \bigoplus_{\lambda \in B_{d,\ell-1}} \gamma_\lambda \mapsto (-  \bigoplus_{\lambda \in B_{d,\ell-1}} i_*  \gamma_\lambda , ~ \sum_{\lambda \in B_{d,\ell-1}} p_{\lambda+1}^* \, \gamma_\lambda), \quad \\
	& g \colon ( \bigoplus_{\lambda \in B_{d,\ell-1}} \alpha_\lambda, ~\varepsilon) \mapsto  \sum_{\lambda \in B_{d,\ell-1}} \pi_{\lambda}^*\alpha_\lambda + j_* \varepsilon,
	\end{align*}
and a left inverse of $f$ is given by 
	$(\bigoplus_{\lambda \in B_{d,\ell-1}} \alpha_\lambda, ~\varepsilon) \mapsto \bigoplus_{\lambda \in B_{d,\ell-1}} p_{\lambda+1\,*} \varepsilon$.
	\item Then the short exact sequence of $(1)$ induces an isomorphism of Chow groups:
	\begin{align*} 
		 \bigoplus_{\lambda \in B_{d,\ell-1}} \pi_\lambda^* \oplus \bigoplus_{\lambda \in B_{d-1,\ell}} \Gamma_\lambda^*  \colon  \bigoplus_{\lambda \in B_{d,\ell-1}} CH_{k-d(\ell-1)+|\lambda|}(X)  \oplus  \bigoplus_{\lambda \in B_{d-1,\ell}} CH_{k-d\ell +|\lambda|}(Z)  \xrightarrow{\sim} CH_{k}(\shH_s).
	\end{align*}
Furthermore, the following relation holds: for any $\lambda, \mu \in B_{d,\ell-1}$, $\nu,\tau \in B_{d-1,\ell}$,
	\begin{align*}
	&\pi_{\lambda *} \, \pi_{\mu}^* = \delta_{\lambda,\mu} \, \Id_{CH(X)}, \quad \Gamma_{\nu *} \, \Gamma_{\tau}^* = (-1)^\ell \delta_{\nu,\tau} \, \Id_{CH(Z)}, \quad \pi_{\lambda *}  \Gamma_{\nu}^*=0, \quad \Gamma_{\nu *} \pi_{\lambda}^*=0; ~~\text{and} 
	\end{align*}	
	\begin{align*}
	\Id_{CH(\shH_s)} =  \sum_{\lambda \in B_{d,\ell-1}} \pi_\lambda^* \,\pi_{\lambda *} + \sum_{\lambda \in B_{d-1,\ell}} \Gamma_\lambda^* \, \Gamma_{\lambda *}.
	\end{align*}	
	\end{enumerate}
\end{theorem}

The theorem can be proved by very similar steps as the Cayley's trick case in \cite[\S 3.1]{J19}, which in turn parallels the proof of blowup case in \cite[\S 6.7]{Ful}. We first show:

\begin{lemma}
\begin{enumerate}
	\item [(a)] (Key formula). For all $\alpha \in CH(X)$, $\lambda \in B_{d,\ell-1}$,
	$$\pi^*_\lambda \, i_* \, \alpha = j_* (\Delta_{(1^d)} \cap p_\lambda^* \alpha)  = j_* \, p_{\lambda+1}^* \, \alpha \in CH(\shH_s).$$
	\item [(b)]  For any $\alpha \in CH(X)$, $\pi_{\lambda *} \, \pi_{\mu}^*  \, \alpha = \delta_{\lambda,\mu} \, \alpha$ for any $\lambda, \mu \in B_{d,\ell-1}$.
	\item [(c)] For $\varepsilon \in CH(G_Z)$, if $j^*\,j_*\,\varepsilon  = 0$ and $p_{\lambda+1\,*} \varepsilon =0$ for all $\lambda \in B_{d,\ell-1}$, then $\varepsilon = 0$
	\item [(d)]  
		\begin{enumerate}
		\item[(i)] For any $\beta \in CH(\shH_s)$ there is an $\varepsilon \in CH(G_Z)$ such that 
			$$\beta = \sum_{\lambda \in B_{d,\ell-1} } \pi_\lambda^*\, \pi_{\lambda\,*} \,\beta + j_* \,\varepsilon.$$
		 \item[(ii)] For any $\beta \in CH(\shH_s)$, if $j^* \beta = 0$, and $\pi_{\lambda \, *} \beta = 0$ for all $\lambda \in B_{d, \ell-1}$, then $\beta=0$.
		 \end{enumerate}
\end{enumerate}
\end{lemma}

\begin{proof}

For (a), since the excess bundle for the diagram (\ref{diagram:Cayley}) is given by $\sV = p^* \sN_i /\sN_j \simeq j^* \shU^\vee = \shU^\vee(G_Z)$, by excess bundle formula (\cite[Thm. 6.3, Prop. 6.2(1) \&6.6]{Ful},
	$$\pi^* \,  i_* (\blank) = j_* \, \pi^!_{G_Z} (\blank) = j_* (c_{d}(\sV) \cap p^*(\blank)) =  j_* (\Delta_{(1^d)} \cap p^*(\blank).$$ 
The results then follows directly from projection formula and Littlewood--Richardson rule.

For (b), the result follows directly from the same properties on $G$ Lem. \ref{lem:Gr:projector}, as $\pi_\lambda^* = \iota^* q_\lambda^*$, $\pi_{\lambda *} = q_{\lambda+1\, *} \iota_*$, and $\iota_* \iota^* (\blank) = c_{d}(\shU^\vee) \cap (\blank) = \Delta_{(1^d)} \cap (\blank)$.

For (c), since $\varepsilon = \sum_{\lambda \in B_{d,\ell}} p_\lambda^* \varepsilon_\lambda$, where $\varepsilon_\lambda = p_{\lambda *} \, \varepsilon \in CH(X)$. By the given condition we know $\varepsilon_{\lambda+1}=0$ for $\lambda \in B_{d,\ell-1}$. Notice that $B_{d,\ell} \backslash \{ \lambda+1 \mid \lambda \in B_{d,\ell-1} \} = B_{d-1,\ell}$. Therefore 
	$$\varepsilon =  \sum_{\lambda \in B_{d-1,\ell}} p_\lambda^* \varepsilon_\lambda  \in \Span \{\Delta_\lambda \}_{\lambda \in B_{d-1,\ell}} = \Span \{\Delta_{\lambda}' \}_{\lambda \in B_{d-1,\ell}},$$
by Lem. \ref{lem:Gr:change_basis}. Since $j^* j_* (\blank) = c_\ell(\shQ) \cap (\blank)  = \Delta_{(\ell)}' \cap (\blank)$ is injective on $\Span \{\Delta_{\lambda}' \}_{\lambda \in B_{d-1,\ell}}$ as $\Delta_{(\ell)}' \cap (\Delta_{\lambda}') = \Delta_{(\lambda^t+1)^t}'$, therefore $j^*j_*(\varepsilon) = 0$ implies $\varepsilon=0$.

For (d) (i), similarly to \cite[\S 3.1]{J19} Step (d)(i), over $U: = X \backslash Z$, the vector bundle $\sG_U^\vee \subset E_U^\vee$ is a linear sub-bundle, $\shH_s|_U = \Gr_d(\sG_U^\vee) \subseteq \Gr_d(E_U^\vee) = G|_U $ is a sub-Grassmannian, which is a locally complete intersection cut out by a regular section $\shU^\vee$, induced canonically by the section $s \in \Gamma(X, E) =  \Gamma(G, \shU^\vee)$. Apply Grassmannian bundle formula to $\shH_s|_U = \Gr_d(\sG_U^\vee)$, every $\beta \in CH(\shH_s|_U )$ can be written as $\beta = \sum_{\lambda \in B_{d,\ell-1}} \pi_\lambda^* \alpha_\lambda$ for some $\alpha_\lambda \in CH(X)$. But then $\pi_{\lambda*} (\beta)= q_{\lambda+1 *} \iota_* (\beta) =  q_{\lambda+1 *} (\Delta_{(1)^d} \cdot \Delta_\lambda \cdot q^* \alpha_\lambda)  = \alpha_{\lambda}$. Hence $\beta - \sum_{\lambda \in B_{d,\ell-1} } \pi_\lambda^*\, \pi_{\lambda\,*} \,\beta =0$ over $U$. Therefore (i) holds by exact sequence $CH(G_Z) \to CH(\shH_s) \to CH(\shH_{s}|_U) \to 0$.

For (d)(ii), by (i) $\beta = j_* \varepsilon$ for some $\varepsilon \in CH(G_Z)$, it then follows that $i_* p_{\lambda+1 \,*} \varepsilon = \pi_{\lambda*} \beta =0$ for all $\lambda \in B_{d,\ell-1}$. Let $\varepsilon = \varepsilon_1 + \varepsilon_2$, where $\varepsilon_2 = \sum_{\lambda \in B_{d,\ell-1}} p_{\lambda+1}^* p_{\lambda+1 \, *} \varepsilon$, and $\varepsilon_1= \sum_{\lambda \in B_{d-1,\ell}} p_{\lambda}^* p_{\lambda *} \varepsilon$ satisfies $p_{\lambda+1 \, *} \varepsilon_1 = 0$ for all $\lambda \in B_{d,\ell-1}$. Then by (a), $j_* \varepsilon_2 = \sum_{\lambda \in B_{d,\ell-1}} \pi_\lambda^* \, i_* \, p_{\lambda+1 \, *} \varepsilon = 0$, hence $\beta = j_* \varepsilon = j_* \varepsilon_1$. Therefore $j^*  j_* \varepsilon_1 = j^* \beta = 0$. By (c), $\varepsilon_1 = 0$ and therefore $\beta =  j_* \varepsilon_1 = 0$. \end{proof}

\begin{proof}[Proof of Thm. \ref{thm:Cayley}] For (1), the result follows easily from above lemma: from the (a) the sequence is a complex, i.e. $g \circ f= 0$; surjectivity follows from (d)(i); the left inverse $h$ of $f$ follows directly from Grassmannian case Lem. \ref{lem:Gr:projector}. For exactness can be proved similar to (d)(ii). In fact suppose there exists $\alpha_\lambda \in CH(X)$, $\lambda \in B_{d,\ell-1}$ and $\varepsilon \in CH(G_Z)$ such that $\sum_{\lambda \in B_{d,\ell-1} } \pi_\lambda^*\, \alpha_\lambda + j_* \,\varepsilon = 0$. Therefore $\alpha_\lambda = - \pi_{\lambda\,*} ( j_* \,\varepsilon ) = - i_* \, p_{\lambda+1\,*} \varepsilon$. Denote $\varepsilon_2 = \sum_{\lambda \in B_{d,\ell-1}} p_{\lambda+1}^* p_{\lambda+1 \, *} \varepsilon$ as (d)(ii), then by (a), $j_* \varepsilon_2 = \sum_{\lambda \in B_{d,\ell-1}} \pi_\lambda^* \, i_* \, p_{\lambda+1 \, *} \varepsilon = - \sum_{\lambda \in B_{d,\ell-1}} \pi_\lambda^* \alpha_\lambda = j_* \varepsilon$. Then $\varepsilon_1: = \varepsilon - \varepsilon_2$ satisfies $j_* \varepsilon_1 = 0$ and $ p_{\lambda+1 \, *} \varepsilon_1 = 0$ for all $\lambda \in B_{d,\ell-1}$, therefore $\varepsilon_1 = 0$ by (c) of above Lemma. Hence $\varepsilon = \varepsilon_2 = \sum_{\lambda \in B_{d,\ell-1}} p_{\lambda+1}^* p_{\lambda+1 \, *} \varepsilon$. This shows the exactness of the sequence.

For (2), the relations among $\pi_{\lambda\,*}$ and $\pi_{\lambda}^*$ is (b) of above Lemma; Similarly, by \ref{lem:Gr:projector} for $G_Z$, 
	\begin{align*}
	\Gamma_{\nu *} \, \Gamma_{\tau}^* (\blank)= p_{(\nu^t+1)^t *}' j_* j^* p_{\tau}'^* (\blank) = p_{(\nu^t+1)^t *}' (c_{\ell}(\shQ^\vee)\cap p_{\tau}'^* (\blank) )  \\
	= p_{(\nu^t+1)^t *}' (\Delta_{(\ell)}' \cdot p_{\tau}'^* (\blank) ) =  (-1)^\ell  p_{(\nu^t+1)^t *}'  p_{(\tau^t+1)^t}'^* (\blank) = \delta_{\nu,\tau} \Id.
	\end{align*}
For orthogonality, notice that for all $\nu \in B_{d-1,\ell}$, by the flatness of ambient square of (\ref{diagram:Cayley}),
	\begin{align*}
	\iota_* \, j_* \, p_{\nu}'^* (\blank)  = q_{\nu}'^* i_* (\blank) \in \Span \{\Delta_{\lambda}' \}_{\lambda \in B_{d-1,\ell}} = \Span \{\Delta_{\lambda} \}_{\lambda \in B_{d-1,\ell}} \subseteq CH(G).
	\end{align*}
hence $\pi_{\lambda *}  \Gamma_{\nu}^* = q_{\lambda+1 \,*} \iota_* \, j_* \, p_\nu'^* =0$ for any $\lambda \in B_{d,\ell-1}$. 
Similarly, for all $\lambda \in B_{d,\ell-1}$, 
	\begin{align*}
	j^* \,  \pi_{\lambda}^* (\blank)= p_{\lambda}^* \,i^* (\blank) \in \Span \{\Delta_\lambda \}_{\lambda \in B_{d,\ell-1}} = \Span \{\Delta_{\lambda}' \}_{\lambda \in B_{d, \ell-1}} \subseteq CH(G_Z).
	\end{align*}
Since $B_{d,\ell} \backslash B_{d, \ell-1} = \{ (\nu^t+1)^t \mid \lambda \in B_{d-1,\ell} \}$, $\Gamma_{\nu *} \pi_{\lambda}^* =p_{(\nu^t+1)^t \,*}' j^* \,  \pi_{\lambda}^* (\blank) = 0$ for any $\nu \in B_{d-1,\ell}$. 

Now the desired decomposition of $\Id_{CH(\shH_s)}$ follows from the exact sequence of statement (1). In fact, for any $\beta \in CH(\shH_s)$, there exists $\alpha_\lambda$ and $\varepsilon$ such that  $\beta = \sum_{\lambda \in B_{d,\ell-1} } \pi_\lambda^*\, \alpha_\lambda + j_* \,\varepsilon$. By replacing $\varepsilon$ by $\varepsilon_1 := \varepsilon -  \sum_{\lambda \in B_{d,\ell-1}} p_{\lambda+1}^* p_{\lambda+1 \, *} \varepsilon$ and $\alpha_\lambda$ by $\alpha_\lambda + i_* (p_{\lambda+1\,*} \varepsilon)$, we may assume 
	$\varepsilon \in \Span \{\Delta_\lambda \}_{\lambda \in B_{d-1,\ell}} = \Span \{\Delta_{\lambda}' \}_{\lambda \in B_{d-1,\ell} }\subseteq CH(G_Z).$
Hence $\varepsilon = \sum_{\nu \in B_{d-1,\ell}} p_{\nu}'^* \varepsilon_\nu$ for some $\varepsilon_\nu \in CH(Z)$, and $\beta = \sum_\lambda \pi_\lambda^*\, \alpha_\lambda +  \sum_{\nu \in B_{d-1,\ell}} \Gamma_{\nu}^* \varepsilon_\nu$. Now it follows from above orthogonality relations that $\alpha_\lambda = \pi_{\lambda_*} \beta$, $\varepsilon_\nu = \Gamma_{\nu *} \beta$. Hence $\beta = \sum \pi_\lambda^*\, \pi_{\lambda_*} \beta  +  \sum_{\nu} \Gamma_{\nu}^*  \Gamma_{\nu *} \beta$. \end{proof}

Similarly as \cite{J19}, it follows immediately from Manin's identity principle that:
\begin{corollary} \label{cor:Cayley} If $X$, $\shH_s = \foQuot_d(\sG)$ and $Z$ are smooth and projective over some ground field $\kk$, then there is an isomorphism of Chow motives over $\kk$:
	\begin{align*} \label{eqn:cayley.chow.motif}
	 \bigoplus_{\lambda \in B_{d,\ell-1}} \pi_\lambda^* \oplus \bigoplus_{\lambda \in B_{d-1,\ell}} \Gamma_\lambda^* \colon  \bigoplus_{\lambda \in B_{d,\ell-1}} \foh(X)(d(\ell-1) - |\lambda|)\oplus  \bigoplus_{\lambda \in B_{d-1,\ell}} \foh(Z)(d\ell - |\lambda|) \xrightarrow{\sim} \foh(\shH_s),
	\end{align*}
where recall $\ell = n-d = \rank \sG - d +1$.
\end{corollary}

\begin{example} \label{ex:cayley.blowup}
	\begin{enumerate}[leftmargin = *]
		\item ({\em Cayley's trick}). If $d=1$, $\ell =n-1=\rank \sG$, in this case $\shH_s = \PP(\sG) : =  \Proj \Sym_{\sO_X}^\bullet  \sG \subset  G = \PP_X(E) := \Proj \Sym_{\sO_X}^\bullet  E$ is a hypersurface, $B_{d, \ell-1} = [0, n-2]$ and $B_{d-1,\ell} =\{0\}$. Hence Thm. \ref{thm:Cayley} implies the formula of {usual Cayley's trick} \cite[\S 3.1]{J19}: 
			$$ \bigoplus_{i = 0}^{n-2} \pi_i^* \oplus (j_* \circ p^*) \colon   \bigoplus_{i = 0}^{n-2}CH_{k-(n-2)+i}(X) \oplus CH_{k-(n-1)}(Z) \simeq CH_k(\shH_s),$$
		where $ \pi_i^* = c_1(\sO(1))^i \cap \pi^*(\blank)$. Similarly for motives, Cor. \ref{cor:Cayley} becomes:
		$$ \bigoplus_{i = 0}^{n-2} \pi_i^* \oplus (j_* \circ p^*) \colon  \bigoplus_{i = 0}^{n-2} \foh(X)(n-2-i) \oplus \foh(Z)(n-1) \simeq \foh(\shH_s).$$
		\item ({\em Blowup}). If $\ell = 1$, $d = n-1$, $n = \codim (Z \subset X)$. Then $\shH_s = \Bl_Z X$ is the {\em blowup}, $G_Z = \PP_Z(E^\vee)  = \PP_{Z, \rm sub}(N_{Z/X})$, $B_{d,\ell-1} = \{0\}$, $B_{d-1,\ell} = \{0, 1^t, 2^t, \ldots, (n-2)^t\}$. Hence Thm. \ref{thm:Cayley} implies the {blowup formula} \cite{Ful}:
			$$\pi^* \oplus \bigoplus_{i=0}^{n-2} \Gamma_{(i^t)}^* \colon CH_{k}(X) \oplus \bigoplus_{i=0}^{n-2} CH_{k-(n-1)+i}(Z) \simeq CH_k( \Bl_Z X),$$
	where $i \in [0, n-2]$,  $\Gamma_{(i^t)}^* = j_* (c_1(\sO_{ \PP_{\rm sub}(N_{Z/X})}(1))^i \cap p^*(\blank))$. Similarly, Cor. \ref{cor:Cayley} becomes:
		$$ \pi^* \oplus \bigoplus_{i=0}^{n-2} \Gamma_{(i^t)}^*  \colon \foh(X) \oplus   \bigoplus_{i = 0}^{n-2} \foh(Z)(n-1-i) \simeq \foh(\Bl_Z X).$$	
	\end{enumerate}
\end{example}

\subsection{Virtual Grassmannian flips} \label{sec:virtual.flips}
Let $V, W$ be vector bundles of rank $n,m$ on a scheme $Z$, $n \ge m$. Let $0 \le d_- \le d_+$ and $0 \le \ell_- \le \ell_+$ be integers such that $d_- + \ell_- = m$, $d_+ + \ell_+ = n$. Denote $\delta_d: = d_+ - d_-$, $\delta_\ell : = \ell_+ - \ell_-$, then $\delta_d + \delta_\ell = n - m = : \delta$. Consider $G_+ : = \Gr_{d_+}(V)$ and $G_- : = \Gr_{d_-}(W)$, and denote $\shU_\pm$, $\shQ_\pm$ the corresponding universal bundles. Let $\Gamma_Z: = G_+ \times_Z G_-$, and denote $r_\pm \colon \Gamma_Z \to G_\pm$ the projection. Therefore we have a Cartesian diagram:
	\begin{equation*}
		\begin{tikzcd}[row sep= 2.6 em, column sep = 2.6 em]
	\Gamma_Z = G_+ \times_Z G_-  \ar{d}[swap]{r_-} \ar{r}{r_+} &G_+ \ar{d}{\pi_+} 
	\\
	G_- \ar{r}{\pi_-}         &Z
		\end{tikzcd}	
		\end{equation*}	 

\begin{lemma}[Virtual Grassmannian flips]  \label{lem:virtual:flip}
\begin{enumerate}
	\item For any $\nu \in B_{\delta_d, \delta_\ell}$ and any $k \in \ZZ$, denote $\nu^c \in B_{\delta_d, \delta_\ell}$ the complement of $\nu$ inside $B_{\delta_d, \delta_\ell}$, consider the following maps:
	\begin{align*}
		&\Psi^{\nu}: = \Psi^{\nu}_{(d_-,d_+)}: =r_{+\,*} (c_{\rm top}(\shQ_-^\vee \otimes \shU_+^\vee) \cup \Delta_{\nu}(-\shU_+) \cap r_-^*(\blank)) \colon		& CH_{k-d_+ \cdot\delta_\ell + |\nu|}(G_-) \to CH_{k}(G_+) ;\\
		&\Psi_{\nu}^{\rm std}: =\Psi^{\rm std}_{(d_-,d_+), \,\nu} := r_{-\,*} (c_{\rm top}(\shQ_+^\vee \otimes \shU_-^\vee) \cup \Delta_{\nu^c}(\shQ_+) \cap r_+^*(\blank)) \colon	&  CH_{k}(G_+) \to CH_{k-d_+ \cdot\delta_\ell + |\nu|}(G_-).
	\end{align*}
Then $\Psi_{\nu}^{\rm std} \circ \Psi^{\nu} = (-1)^{d_- \cdot \delta_\ell} \Id_{CH(G_-)}$. In particular, $\Psi^{\nu}$ is injective. 
	\item For any fixed $\nu_{\rm fix} \in B_{\delta_d, \delta_\ell}$, the map $\bigoplus_{\nu \subseteq \nu_{\rm fix}} \Psi^{\nu}$ is injective, with image
		\begin{align*}
		\Im ( \bigoplus_{\nu \subseteq \nu_{\rm fix}} \Psi^{\nu}) & =  \Span \{ \Delta_{\tau}(-\shU_+) \mid (\ell_-^{\delta_d}) \subseteq \tau \subseteq \nu_{\rm fix} + (\ell_-^{d_+}), ~\tau \in B_{d_+,\ell_-}\}  \\
		& =  \Span \{ \Delta_{\lambda \oslash \nu}(-\shU_+) \mid \lambda \in B_{d_-,\ell_-}, \nu \in B_{\delta_d,\delta_\ell}, \nu \subseteq \nu_{\rm fix} \} \subseteq CH(G_+).
		\end{align*}
\end{enumerate}
\end{lemma}
Recall $\lambda \oslash \nu$ denotes the ordinary Young diagram obtained by placing $\lambda$ below the rectangle $(\delta_d \times  \ell_-)$ and placing $\nu$ to the right, i.e. of shape ${\scriptsize \young(*\nu,\lambda)}$, see \S \ref{sec:Young}.

\begin{proof} (1). The map $\Psi_{\nu}^{\rm std} \circ \Psi^{\nu} $ is given by the convolution of the correspondences 
	\begin{align} \label{eqn:virtual:convolution}
	\begin{split}
	&(c_{\rm top}(\shQ_+^\vee \otimes \shU_-^\vee)  \cup \Delta_{\nu^c}(\shQ_+) ) * (c_{\rm top}(\shQ_-^\vee \otimes \shU_+^\vee)  \cup \Delta_{\nu}(-\shU_+) ) \\
	& = p_{13*} \Big(p_{12}^* \big (c_{\rm top}(\shQ_-^\vee \otimes \shU_+^\vee)  \cdot \Delta_{\nu}(-\shU_+) \big)  \cdot  p_{23}^* \big(c_{\rm top}(\shQ_+^\vee \otimes \shU_-^\vee)  \cdot \Delta_{\nu^c}(\shQ_+)\big) \Big)
	\end{split}
	\end{align}
in $CH(G_- \times_Z G_-)$, where $p_{ij}$'s are the obvious projections from $G_- \times_Z G_+ \times_Z G_-$. To avoid confusion we denote the product by $G_-^{(1)} \times_Z G_+ \times_Z G_-^{(2)}$. It follows from Lem. \ref{lem:tensor} that:
	\begin{align*}
		 & c_{\rm top}(\shQ_-^{(1)\vee} \otimes \shU_+^\vee)   = \sum_{\lambda \in B_{d_+, \ell_-}} \Delta_{\lambda^t}(\shU_+^\vee) \cdot \Delta_{\lambda^{c_{+-}}}(\shQ_-^{(1)\vee}) = \sum_{\lambda \in B_{d_+, \ell_-}} \Delta_{\lambda}(-\shU_+) \cdot \Delta_{\lambda^{c_{+-}}}(\shQ_-^{(1)\vee}), \\
		& c_{\rm top}(\shQ_+^\vee \otimes \shU_-^{(2)\vee})  = (-1)^{\ell_+ d_-} c_{\rm top}(\shQ_+ \otimes \shU_-^{(2)}) =  (-1)^{\ell_+ d_-} \sum_{\mu \in B_{d_-, \ell_+}} \Delta_{\mu^t}(\shU_-^{(2)}) \cdot \Delta_{\mu^{c_{-+}}}(\shQ_+).
	\end{align*}
Here we use $(\blank)^{c_{+-}}$ or $(\blank)^{c_{-+}}$ to distinguish taking complements inside the box $B_{d_+,\ell_-}$ or $B_{d_-,\ell_+}$. Hence the convolution $(-1)^{\ell_+ d_-} \cdot$ (\ref{eqn:virtual:convolution}) is:
	\begin{align*}
	&    p_{13*} \Big( \sum_{\substack{\lambda \in B_{d_+, \ell_-}, \\ \mu \in B_{d_-, \ell_+}}}  \Delta_{\lambda^{c_{+-}}}(\shQ_-^{(1)\vee})  \Delta_{\mu^t}(\shU_-^{(2)}) \big( \Delta_{\lambda}(-\shU_+) \Delta_{\nu}(-\shU_+)   \Delta_{\mu^{c_{-+}}}(\shQ_+)  \Delta_{\nu^{c_{\delta,\delta}}}(\shQ_+) \big) \Big) \\
	&  =   \sum_{\substack{\lambda \in B_{d_+, \ell_-}, \\ \mu \in B_{d_-, \ell_+}}}  \Delta_{\lambda^{c_{+-}}}(\shQ_-^{(1)\vee})  \Delta_{\mu^t}(\shU_-^{(2)}) \cdot \int_{G_+} \Delta_{\lambda}(-\shU_+) \Delta_{\nu}(-\shU_+)   \Delta_{\mu^{c_{-+}}}(\shQ_+)  \Delta_{\nu^{c_{\delta,\delta}}}(\shQ_+).
	\end{align*}
(Here $(\blank)^{c_{\delta,\delta}}$ denotes taking complements inside the box $B_{\delta_d,\delta_\ell}$.) By Littlewood--Richardson rule Lem. \ref{lem:LR} and duality Lem. \ref{lem:Gr:dual}, the latter integration over $G_+$ is equal to the summation of products of Littlewood--Richardson coefficients:
	\begin{align} \label{eqn:proof:virtual}
	\sum_{\tau \in B_{d_+,\ell_+}} c_{\lambda,\nu}^{\tau} \cdot c_{\mu^{c_{-+}},\nu^{c_{\delta,\delta}}}^{\tau^{c_{+,+}}}. 
	\end{align}
We claim this quantity is zero unless $\lambda = (\alpha^t+(\delta_d^{\ell_-}))^t$ and $\mu = \alpha$ for some $\alpha \in B_{d_-,\ell_-}$. Assume it is nonzero. The following diagram may help understand the situation:
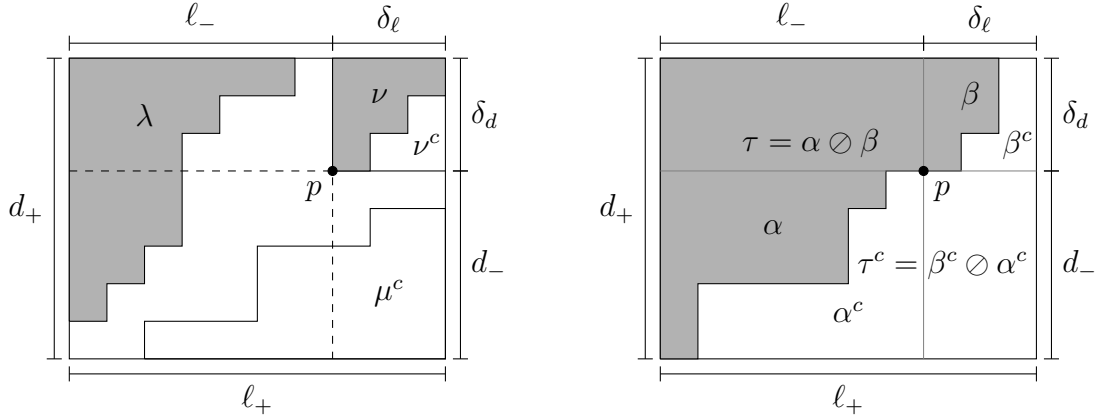
\begin{center}
\begin{tikzpicture}[scale=0.5]
    \draw (0,0) rectangle (10,8);
    \draw (7,8) -- (7,5) -- (10,5);
    
    \draw (-0.4,0) -- (-0.4,8);
    \draw (-0.2,0) -- (-0.6,0);
    \draw (-0.6,8) -- (-0.2,8);
    \node [left] at (-0.4,4) {$d_+$};
    
     \draw (10.4,0) -- (10.4,5);
    \draw (10.6,0) -- (10.2,0);
    \draw (10.6,5) -- (10.2,5);
    \node [right] at (10.4,2.5) {$d_-$};
    
     \draw (10.4,5) -- (10.4,8);
    \draw (10.2,5) -- (10.6,5);
    \draw (10.2,8) -- (10.6,8);
    \node [right] at (10.4, 6.5) {$\delta_d$};
    
     \draw (0,-0.4) -- (10,-0.4);
    \draw (0,-0.2) -- (0,-0.6);
    \draw (10,-0.2) -- (10,-0.6);
    \node [below] at (5,-0.4) {$\ell_+$};
    
      \draw (0,8.4) -- (7,8.4);
    \draw (0,8.2) -- (0,8.6);
    \draw (7,8.2) -- (7,8.6);
    \node [above] at (3.5,8.4) {$\ell_-$};
    
      \draw (7,8.4) -- (10,8.4);
    \draw (7,8.2) -- (7,8.6);
    \draw (10,8.2) -- (10,8.6);
    \node [above] at (8.5, 8.4) {$\delta_{\ell}$};
    
    \draw [fill=light-gray] (0,1) -- (0,8) -- (6,8) -- (6,7) -- (4,7) -- (4,6)  -- (3,6)  -- (3,3) -- (2,3) -- (2,2) -- (1,2) -- (1,1) -- (0,1);
    \node at (2,6.5) {$\lambda$};
    
        \draw (2,0) -- (10,0) -- (10,4) -- (8,4) -- (8,3) -- (5,3) -- (5,1) -- (2,1) -- (2,0);
    \node at (8.5, 1.75) {$\mu^c$};
    
       \draw [fill=light-gray] (7,5) -- (7,8) -- (10,8) -- (10,7) -- (9,7) -- (9,6) -- (8,6) -- (8,5) -- (7,5);
    \node at (8.25,7) {$\nu$};
    \node at (9.5,5.7) {$\nu^c$};
    
      \draw [dashed] (0,5) -- (7,5) -- (7,0);
      \draw [fill=black] (7,5) circle [radius=0.12];
         \node [below left] at (7,5) {$p$};
\end{tikzpicture}
\qquad 
\begin{tikzpicture}[scale=0.5]
    \draw (0,0) rectangle (10,8);
    
    \draw (-0.4,0) -- (-0.4,8);
    \draw (-0.2,0) -- (-0.6,0);
    \draw (-0.6,8) -- (-0.2,8);
    \node [left] at (-0.4,4) {$d_+$};
    
     \draw (10.4,0) -- (10.4,5);
    \draw (10.6,0) -- (10.2,0);
    \draw (10.6,5) -- (10.2,5);
    \node [right] at (10.4,2.5) {$d_-$};
    
     \draw (10.4,5) -- (10.4,8);
    \draw (10.2,5) -- (10.6,5);
    \draw (10.2,8) -- (10.6,8);
    \node [right] at (10.4, 6.5) {$\delta_d$};
    
     \draw (0,-0.4) -- (10,-0.4);
    \draw (0,-0.2) -- (0,-0.6);
    \draw (10,-0.2) -- (10,-0.6);
    \node [below] at (5,-0.4) {$\ell_+$};
    
      \draw (0,8.4) -- (7,8.4);
    \draw (0,8.2) -- (0,8.6);
    \draw (7,8.2) -- (7,8.6);
    \node [above] at (3.5,8.4) {$\ell_-$};
    
      \draw (7,8.4) -- (10,8.4);
    \draw (7,8.2) -- (7,8.6);
    \draw (10,8.2) -- (10,8.6);
    \node [above] at (8.5, 8.4) {$\delta_{\ell}$};
    
    \draw [fill=light-gray] (0,0) -- (0,8) -- (9,8) -- (9,6) -- (8,6) -- (8,5)  -- (6,5)  -- (6,4) -- (5,4) -- (5,2) -- (1,2) -- (1,0) -- (0,0);
    \node at (4,5.75) {$\tau = \alpha \oslash \beta $};
        \node at (7.5,2.5) {$\tau^c = \beta^c \oslash \alpha^c $};
     \node at (3,3.5) {$\alpha$};
     \node at (5,1.25) {$\alpha^c$};
    \node at (8.25,7) {$\beta$};
    \node at (9.5,5.7) {$\beta^c$};
    
      \draw [help lines] (0,5) -- (7,5) -- (7,0);
          \draw [help lines] (7,8) -- (7,5) -- (10,5);     
           \draw [fill=black] (7,5) circle [radius=0.12];
         \node [below right] at (7,5) {$p$};
        
\end{tikzpicture} 
\end{center}
First, as $\tau \subseteq (\delta_d \times \ell_+) \cup (d_+  \times \ell_-)$, hence $\tau^{c_{++}} \supseteq (d_- \times \delta_\ell)$. Similarly $\tau \supseteq (\delta_d \times \ell_-)$. Therefore $\tau$ has to have the form $\tau = \alpha \oslash \beta$ for $\alpha \in B_{d_-,\ell_-}, \beta \in B_{\delta_d, \delta_\ell}$, and $\tau^{c_{++}} = \beta^{c_{\delta,\delta}} \oslash \alpha^{c_{-,-}}$ (i.e. the path for $\tau$ has to pass through the point $p$ in above diagram). Therefore $\lambda \subseteq (\alpha^t + (\delta_d \times \ell_-)^t)^t$ (i.e. contained in shape ${\scriptsize \young(*,\alpha)}$). Hence $\nu \supseteq \beta$ by Littlewood-Richardson rule. (One way to see this is, as we need to be able to express down $\tau={\scriptsize \young(*\beta,\alpha)}$ as a strict $\nu$-expansion of $\lambda \subseteq {\scriptsize \young(*,\alpha)}$, therefore by Littlewood-Richardson rule \cite{Ful, FulY}, the first row of boxes of $\beta$ has to be filled by ${\scriptsize \young(1)}$ from the first row of $\nu$, and the second row of boxes of $\beta$ has to be filled by ${\scriptsize \young(2)}$ from the second row of $\nu$, etc, hence $\beta \subseteq \nu$.) Similarly $\mu^{c_{-+}} \subseteq (\alpha^{c_{--}} + (d_- \times \delta_\ell))$ and this implies $\nu^{c_{\delta \delta}} \supseteq \beta^{c_{\delta \delta}}$. Therefore $\nu = \beta$. This forces $\lambda = (\alpha^t+(\delta_d^{\ell_-}))^t$ and $\mu^{c_{-+}} = \alpha^{c_{--}} + (\delta_\ell^{d_-})$ i.e. $\mu = \alpha$. In this case (\ref{eqn:proof:virtual}) is $1$, and therefore the convolution (\ref{eqn:virtual:convolution}) becomes:
	\begin{align*}	
	(-1)^{\ell_+ d_-} \sum_{\alpha \in B_{d_-,\ell_-}}  \Delta_{\alpha^{c}}(\shQ_-^{(1)\vee})  \Delta_{\alpha^t}(\shU_-^{(2)})  
	 =  (-1)^{\ell_+ d_-} \cdot  (-1)^{\ell_- d_-}  c_{\rm top} (\shU_-^{(2)\vee} \boxtimes_X \shQ_-^{(1)})  
	\end{align*}
which equals the diagonal $(-1)^{\delta_\ell d_-} \cdot [\Delta_{G_-}]$. Hence $\Psi_{\nu}^{\rm std} \circ \Psi^{\nu} =(-1)^{\delta_\ell d_-}  \Id_{CH(G_-)}$.

(2). As the maps are all ``$CH(Z)$-linear", by Grassmannian bundle formula Thm. \ref{thm:Grassmannian} we need only to consider their actions on basis. First consider the case $\nu_{\rm fix}=(0)$. For $\lambda \in B_{d_-,\ell_-}$, denote $\widetilde{\lambda} : = (\lambda^t + (\delta_d^{\ell_-}))^t$, i.e the Young diagram of shape ${\scriptsize \young(*,\lambda)}$ obtained by placing $\lambda$ below the rectangle $(\delta_d \times  \ell_-)$. Then up to signs $\Psi^{(0)}$ and $\Psi_{(0)}^{\rm std}$ induces a bijection between $\{\Delta_{\lambda}(-\shU_-)\}_{\lambda \in B_{d_-,\ell_-}}$ and $\{\Delta_{\widetilde{\lambda}}(-\shU_+) \}_{\lambda \in B_{d_-,\ell_-}}$:
	$$\Psi^{(0)} \colon \Delta_{\lambda}(-\shU_-) \mapsto \pm \Delta_{\widetilde{\lambda}}(-\shU_+), \qquad \Psi_{(0)}^{\rm std} \colon \Delta_{\widetilde{\lambda}}(-\shU_+) \mapsto \pm \Delta_{\lambda}(-\shU_-).$$
The results clearly hold. Next, for any $\nu$, then $\Psi^\nu$ maps for any $\Delta_{\lambda}(-\shU_-)$ up to a sign to:
	\begin{align*}
	\pm \Psi^\nu (\Delta_{\lambda}(-\shU_-)) =  \Delta_{\widetilde{\lambda}}(-\shU_+) \cdot \Delta_{\nu}(-\shU_+)  
	= \Delta_{\lambda  \oslash \nu}(-\shU_+) + \sum_{ \mu \in B_{d_-,\ell_-}, \,\tau \subsetneq \nu} c_{\widetilde{\lambda},\nu}^{\mu \oslash \tau} \Delta_{\mu \oslash \tau}(-\shU_+),
	\end{align*}
where the second summand belongs to $\Im(\bigoplus_{\tau \subsetneq \nu} \Psi^{\tau})$. Hence up to $\Im(\bigoplus_{\tau \subsetneq \nu} \Psi^{\tau})$ and up to signs, the images of basis $\{\Delta_{\lambda}(-\shU_-)\}_{\lambda \in B_{d_-,\ell_-}}$ of $CH(G_-)$ under the map $\Psi^\nu$ hit exactly each basis $\Delta_{\mu}(-\shU_+)$ once, $\mu \in B_{d_-,\ell_-} \oslash \nu$. As all maps are $CH(Z)$-linear, and the sets $B_{d_-,\ell_-} \oslash \nu$ are disjoint for different $\nu$, inductively we have for any $\nu_{\rm fix}$, the map $\bigoplus_{\nu \subseteq \nu_{\rm fix}} \Psi^{\nu}$ is injective, and the image is the subgroup described by the lemma.
\end{proof}

Next we fix $d_+ := d$, $\ell_+ := \ell = n-d$, and let $d_-$ vary. 
Let $d_- = d - i$, where $0 \le i \le \delta: = n-m$. Then we have a series of maps:
	$$\Psi_{(d-i,d)}^{\nu^{(i)}} \colon CH_{k - d(\delta-i) + |\nu^{(i)}|} (G_{d-i}(W)) \to CH_k( G_{d}(V))$$
for all $i \in  [0,\delta]$ and $\nu^{(i)} \in B_{i, \delta -i}$; similarly for $\Psi^{\rm \,std}_{(d-i,d), \,\nu^{(i)}}$. For the set of indices $(i \in [0,\delta], \nu^{(i)} \in B_{i, \delta -i})$, consider the following partial order:
	\begin{align} \label{eqn:porder}
	  (i, \nu^{(i)}) \underset{ (\preceq)}{\prec} (j, \tau^{(j)})  \overset{\rm def}{\iff} i < j \text{~or~} i=j, \,\nu^{(i)} \underset{ (\subseteq)}{\subsetneq} \tau^{(j)}.
	\end{align}
Then the following is what happens along each stratum for the general $\foQuot$--formula:
\begin{theorem}  \label{thm:main:local}
	\begin{enumerate}
		\item The maps of Lem. \ref{lem:virtual:flip} are semiorthogonal in the following sense: for any pair $(i,  \nu^{(i)} \in B_{i, \delta -i})$ and $(j, \tau^{(j)} \in B_{j, \delta -j})$, the following holds:
				$$\Psi^{\rm \,std}_{(d-j,d), \,\tau^{(j)}} \circ \Psi_{(d-i,d)}^{\nu^{(i)}}  = 0 \qquad \text{if} \qquad (i, \nu^{(i)}) \nsucceq (j, \tau^{(j)}).$$
		\item For any $d \ge 0$, $k \ge 0$, there is an isomorphism 
	\begin{align*}
	\Psi^*:=\bigoplus_{i=0}^\delta \bigoplus_{\nu^{(i)} \in B_{i,\delta-i}} \Psi_{(d-i,d)}^{\nu^{(i)}}  \colon  \bigoplus_{i=0}^\delta \bigoplus_{\nu^{(i)} \in B_{i,\delta-i}} CH_{k - d(\delta-i) + |\nu^{(i)}|} (G_{d-i}(W)) \xrightarrow{\sim} CH_k( G_{d}(V)).
	\end{align*}
	\end{enumerate}
\end{theorem}
On the level Young diagram, the theorem corresponds to, for any given fixed $\delta \ge 0$, a decomposition of the set $B_{d,\ell}$ of Young diagrams into disjoint unions:
	\begin{align*}
	B_{d,\ell} =  \bigsqcup_{i = 0}^{\delta} B_{d-i,\ell - \delta+i} \oslash  B_{i,\delta - i},  \quad \text{where} \quad B_{1} \oslash B_{2} : = \{\lambda \oslash \nu \mid \lambda \in B_{1}, \, \nu \in B_{2} \}.
	\end{align*}
(Here the $\oslash$ are taken with respect to each pair ($B_1$, $B_2$), see \S \ref{sec:Young}.)  The situation is illustrated in the following diagram:
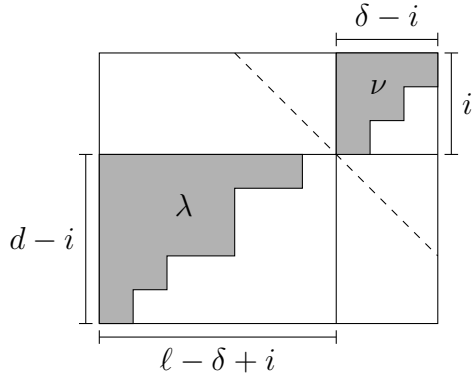
\begin{center}
\begin{tikzpicture}[scale=0.45]
    \draw (0,0) rectangle (10,8);
      \draw [dashed] (4,8) -- (10,2);
    \draw (0,5) -- (7,5) -- (7,0);
    \draw (7,8) -- (7,5) -- (10,5);
    
     \draw (-0.4,0) -- (-0.4,5);
    \draw (-0.6,0) -- (-0.2,0);
    \draw (-0.6,5) -- (-0.2,5);
    \node [left] at (-0.4,2.5) {$d-i$};
    
       \draw (0,-0.4) -- (7,-0.4);
    \draw (0,-0.2) -- (0,-0.6);
    \draw (7,-0.2) -- (7,-0.6);
    \node [below] at (3.5,-0.4) {$\ell - \delta+i$};
    
      \draw (7,8.4) -- (10,8.4);
    \draw (7,8.2) -- (7,8.6);
    \draw (10,8.2) -- (10,8.6);
    \node [above] at (8.5, 8.4) {$\delta-i$};
    
      \draw (10.4,5) -- (10.4,8);
    \draw (10.2,5) -- (10.6,5);
    \draw (10.2,8) -- (10.6,8);
    \node [right] at (10.4, 6.5) {$i$};
    
    \draw [fill=light-gray] (0,0) -- (0,5) -- (6,5) -- (6,4) -- (4,4) -- (4,2) -- (2,2) -- (2,1) -- (1,1) -- (1,0) -- (0,0);
    \node at (2.5,3.5) {$\lambda$};
    
       \draw [fill=light-gray] (7,5) -- (7,8) -- (10,8) -- (10,7) -- (9,7) -- (9,6) -- (8,6) -- (8,5) -- (7,5);
    \node at (8.25,7) {$\nu$};
\end{tikzpicture}
\end{center}

\begin{proof} (1) Denote $\nu:=\nu^{(i)}, \tau:=\tau^{(j)} $, we need to show $\Psi^{\rm \,std}_{(d-j,d), \tau} \circ \Psi_{(d-i,d)}^{\nu} =0$ if $i < j$ or $i = j, \nu \nsupseteq  \tau$. From the proof of Lem. \ref{lem:virtual:flip}, the map $\Psi^{\rm \,std}_{(d-j,d),\,\tau}$ sends $\Delta_{\lambda}(-\shU_+)$, $\lambda \in B_{d,\ell}$ to:
	\begin{align*}
	\pm \sum_{\mu \in B_{d-j, \ell}} \Delta_{\mu}(-\shU_-) \cdot \int_{G_+} \Delta_{\lambda}(-\shU_+)  \Delta_{\mu^{c_{-+}}}(\shQ_+)  \Delta_{\tau^{c_{\delta,\delta}}}(\shQ_+) = \pm \sum_{\mu \in B_{d-j, \ell}} c_{\mu^{c_{-+}}, \tau^{c_{\delta,\delta}}}^{\lambda^c} \cdot \Delta_{\mu}(-\shU_-)
	\end{align*}
Similar to the analysis from the proof of Lem. \ref{lem:virtual:flip}, if the Littlewood-Richardson coefficient is nonzero, then firstly $\lambda^c \subseteq (d \times(\delta-j)) \cup ((d-j)\times \ell)$, hence $\lambda \supseteq (\ell \times j)$. From the description of image of Lem. \ref{lem:virtual:flip} (2), this implies $\Psi^{\rm \,std}_{(d-j,d), \tau} \circ \Psi_{(d-i,d)}^{\nu} =0$ for any $i < j$. Next consider $i=j$, and $\lambda = \alpha \oslash \beta$. By the same argument of Lem. \ref{lem:virtual:flip}, from Littilewood--Richardson rule $c_{\mu^{c_{-+}}, \tau^{c_{\delta,\delta}}}^{\lambda^c} \ne 0$ implies $\beta^{c_{\delta,\delta}} \subseteq \tau^{c_{\delta,\delta}}$ i.e. $\beta \supseteq \tau$. Together with Lem. \ref{lem:virtual:flip} (2), this implies $\Psi^{\rm \,std}_{(d-j,d), \tau} \circ \Psi_{(d-j,d)}^{\nu} =0$ whenever $\nu \nsupseteq \tau$.

(2) The above semi-orthogonality property allows us to easily construct an inverse of the map of $\Psi^*: = \bigoplus_{i=0}^\delta \bigoplus_{\nu^{(i)} \in B_{i,\delta-i}} \Psi_{(d-i,d)}^{\nu^{(i)}}$ by induction, similarly to the case of \cite{Manin}. We define a series of maps
	$\Psi_{(d-i,d), \,\nu^{(i)}}  \colon CH_{k}(\Gr_d(V)) \to CH_{k-d(\delta-i) + |\nu^{(j)}|}(\Gr_{d-i}(W))$
as follows:
\begin{itemize}
	\item In base case, there are two possibilities. If $d > \delta$, then the maximal index is $i_{\rm max} = \delta$ and $\nu_{\rm max}^{(\delta)} =(0)$. In this case, define 
		$$\Psi_{(d-\delta,d), \,\nu_{\rm max}^{(\delta)}} :=  \Psi^{\rm \,std}_{(d-\delta,d), \,(0)} =  r_{(d-\delta,d), -\,*} (c_{\rm top}(\shQ_+^\vee \otimes \shU_{(d-\delta,d), -}^\vee) \cap r_{(d-\delta,d),+}^*(\blank)).$$ 
	If $d \le \delta$, then the maximal index is $i_{\rm max} = d$ and $\nu_{\rm max}^{(d)} = ((\delta-d)^d)$, then define 
		$$\Psi_{(0,d), \,\nu_{\rm max}^{(d)}} :=  \Psi^{\rm \,std}_{(0,d), \,((\delta-d)^d)} =  r_{(0,d), -\,*} (c_{\rm top}(\shQ_+^\vee \otimes \shU_{(0,d),-}^\vee) \cap r_{(0,d),+}^*(\blank)).$$
	\item Assume $\Psi_{(d-j,d), \,\tau^{(j)}}$ are defined for all $(j, \tau^{(j)}) \succ (i, \nu^{(i)})$, then define 
		$$\Psi_{(d-i,d), \,\nu^{(i)}} : = \Psi^{\rm std}_{(d-i,d), \,\nu^{(i)}} \circ \left(\Id -  \sum_{(j, \tau^{(j)}) \succ (i, \nu^{(i)})} (-1)^{(d-j)(\delta-j)}  \, \Psi_{(d-j,d)}^{\tau^{(j)}} \, \Psi_{(d-j,d), \,\tau^{(j)}} \right).$$
\end{itemize}
From semiorthogonality (1) and Lem. \ref{lem:virtual:flip}, it is direct to verify that 
	$$\Psi_* :=\left((-1)^{(d-i)(\delta-i)}  \, \Psi_{(d-i,d), \,\nu^{(i)}}\right)_{(i, \nu^{(i)})}$$
 is the inverse of $\Psi^*$. 
\end{proof}

\section{Main results}

Let $\sG$ be a coherent sheaf of homological dimension $\le 1$ on $X$, and denote
	$$\delta: = \rank \sG \quad \text{and} \quad \sK: = \sExt^1_{X}(\sG,\sO_X) \in \Coh (X).$$
Fix a positive integer $d\in [1,\delta]$. For any $j \in [0,d]$, consider the following schemes:
	\begin{align*}
	 \shZ_{d}^+ : = \foQuot_{X,d}(\sG) \quad \text{and} \quad \pi_{d-j}^- \colon \shZ_{d-j}^- : = \foQuot_{X,d - j}(\sK), ~j\in[0,d].
	\end{align*}  
Denote $\Gamma_{(d-j,d)}:= \shZ_{d-j}^- \times_X \shZ_{d}^+$ the fiber product, and consider the fibered diagram:
	\begin{equation} \label{diagram:Gamma}
	\begin{tikzcd}[row sep= 3 em, column sep = 5 em]
		\Gamma_{(d-j,d)} \ar{r}{r^+_{(d-j,d)}}  \ar{d}[swap]{r^-_{(d-j,d)}} &  \shZ_{d}^+ \ar{d}{\pi}\\
		\shZ_{d-j}^- \ar{r}{\pi'_{d-j}} & X
	\end{tikzcd}
	\end{equation}
Let $\shU := \sQ_d^\vee$ be dual of the tautological rank $d$ locally free quotient on $\shZ_{d}^+$, i.e. $\pi^* \sG \twoheadrightarrow \shU^\vee$ is the tautological quotient sequence of \S \ref{sec:Quot} for the Quot-scheme $\foQuot_{X,d}(\sG)$. \footnote{The reason for the dual notation is that if $\sG = E$ is locally free, then $\shZ_d^+ = \Gr_d(E^\vee)$ and $\shU = \shU_d(E^\vee)$ is the tautological rank $d$ subbundle on the Grassmannian.}
Then for any $j \in [0,d]$, $\nu^{(j)} \in B_{j,\delta-j}$, $k \in \ZZ$, consider the following map:
	$$\Gamma_{(d-j,d)}^{\nu^{(j)}}: =r^+_{(d-j,d)*}(\Delta_{\nu^{(j)}}(-\shU) \cap r^{-\,*}_{(d-j,d)} (\blank) ) \colon CH_{k - d(\delta-j)+|\nu^{(j)}|} (\shZ_{d-j}^-) \to CH_k(\shZ_{d}^+).$$

The goal of this chapter is the prove the following theorem:
\begin{theorem} \label{thm:main} Let $X$ be a Cohen--Macaulay scheme of pure dimension, and $\sG$ be a coherent sheaf of rank $\delta \ge 0$ on $X$ of homological dimension $\le 1$. Assume $X^{\ge \delta+i}(\sG)$ is reduced, and
	$${\rm codim} (X^{\ge \delta+i}(\sG) \subset X) \ge i(\delta+i) \quad \text{ for all} \quad i \ge 1.$$ Then for any $k \ge 0$, there is an isomorphism of Chow groups:
		\begin{align}\label{eqn:main.thm}
	\Gamma^*:=\bigoplus_{j=0}^{\min\{\delta,d\}} \bigoplus_{\nu^{(j)} \in B_{j,\delta-j}} \Gamma_{(d-j,d)}^{\nu^{(j)}} \colon  \bigoplus_{j=0}^{\min\{\delta,d\}}  \bigoplus_{\nu^{(j)} \in B_{j,\delta-j}} CH_{k - d(\delta-j)+|\nu^{(j)}|} (\shZ_{d-j}^-)  \xrightarrow{\sim} CH_k(\shZ_{d}^+).
	\end{align}
\end{theorem}
 
 As mentioned in the introduction, there are in general two types of behaviours of the $\shZ_d^+$:
\begin{itemize}
 		\item If $d \le  \delta$, then $\pi \colon \shZ_d^+ \to X$ is generically a Grassmannian $\Gr_{d}(\delta)$-bundle. 
		\item If $d >  \delta$, then $\shZ_d^+$ is supported over $X^{\ge d}(\sG)$. In fact $\shZ_d^+ \dashrightarrow \shZ_{d-\delta}^-$ is a flip (or a flop if $\delta=0$), both of them map birationally onto $X^{\ge d}(\sG)$.
\end{itemize}

As above construction commutes with base-change, by Manin's identity principle \cite{Manin}:
\begin{corollary} \label{cor:main.motif} In above situation, if $X$, $\shZ_{d-j}^-$, $j \in [0,d]$ and $\shZ_d^{+}$ are smooth and projective over $\kk$, then there is an isomorphism of covariant integral Chow motives over $\kk$:
	\begin{align*} \label{eqn:cayley.chow.motif}
	\bigoplus_{j=0}^{\min\{\delta,d\}}  \bigoplus_{\nu^{(j)} \in B_{j,\delta-j}} [\Gamma_{(d-j,d)}^{\nu^{(j)}}] \colon  \bigoplus_{j=0}^{\min\{\delta,d\}}  \bigoplus_{\nu^{(j)} \in B_{j,\delta-j}}  \foh(\shZ_{d-j}^-) (d(\delta-j)-|\nu^{(j)}|) \xrightarrow{\sim} \foh(\shZ_{d}^+).
	\end{align*}
\end{corollary}
	
\begin{remark} For the readers' convenience, let us also state the results in contravariant setting. The isomorphism of Chow group is: for any $k \in \ZZ$,
	\begin{align*}
	\bigoplus_{j=0}^{\min\{\delta,d\}}  \bigoplus_{\nu^{(j)} \in B_{j,\delta-j}} \Gamma_{(d-j,d)}^{\nu^{(j)}} \colon  \bigoplus_{j=0}^{\min\{\delta,d\}}  \bigoplus_{\nu^{(j)} \in B_{j,\delta-j}} CH^{k - (d-j)(\delta-j)-|\nu^{(j)}|} (\shZ_{d-j}^-)  \xrightarrow{\sim} CH^k(\shZ_{d}^+).
	\end{align*}
If we use $h$ to denote the contravariant Chow motive and $L$ the Lefschetz motif, then:
	\begin{align*} 
	\bigoplus_{j=0}^{\min\{\delta,d\}}  \bigoplus_{\nu^{(j)} \in B_{j,\delta-j}} [\Gamma_{(d-j,d)}^{\nu^{(j)}}] \colon  \bigoplus_{j=0}^{\min\{\delta,d\}}  \bigoplus_{\nu^{(j)} \in B_{j,\delta-j}}  h(\shZ_{d-j}^-) \otimes L^{(d-j)(\delta-j)+|\nu^{(j)}|} \xrightarrow{\sim} h(\shZ_{d}^+).
	\end{align*}
\end{remark}

\subsection{Proof of main theorem}

The following is the direct analogue of Lem. 4.9 of \cite{J19}.

\begin{lemma}  \label{lem:stratum} Assume $\sG$ is a coherent sheaf on $X$ of homological dimension $\le 1$ and rank $r$. For a fixed integer $i \ge 0$, assume $\sG$ has constant rank $r+i+1$ over a reduced locally complete intersection subscheme $Z \subset X$ of codimension $(i+1)(r+i+1)$, and has rank $\le r+i$ over $X \backslash Z$. Denote $\sK = \sExt^1(\sG, \sO)$,  and $i \colon Z \hookrightarrow X$ the inclusion, $G_Z: =   i^* \sG$, $K_Z : = i^* \sK$ are vector bundles over $Z$ of rank $r+i+1$ and $i+1$ respectively. Let $d_-,d_+$ be integers such that $0 \le d_- \le i+1$, $0 \le d_+ \le r+i+1$ and $d_+ - r \le d_- \le d_+$. Denote $G_+: = \Gr_{d_+}(G_Z^\vee)$, $G_- := \Gr_{d_-}(K_Z^\vee)$, and $\shU_{\pm}$, $\shQ_{\pm}$ the corresponding universal subbundles and quotient bundles, and denote $\shZ_+: = \foQuot_{d_+}(\sG)$, $\shZ_-: = \foQuot_{d_-}(\sK)$. Consider the following base-change diagram for the fibered product $\Gamma: = \shZ_+ \times_X \shZ_-$, with names of maps as indicated:
	\begin{equation} \label{diagram:strata}
	\begin{tikzcd}[back line/.style={}, row sep=1.2 em, column sep=2.6 em]
& \Gamma_Z = G_+ \times_Z G_- \ar[back line]{dd}[near start]{r_{Z-}} \ar[hook]{rr}{\ell} \ar{ld}[swap]{r_{Z+}}
  & & \Gamma : = \shZ_+ \times_X \shZ_- \ar{dd}{r_-} \ar{ld}{r_+} \\
G_+ \ar{dd}[swap]{\pi_Z}  
  & & \shZ_+   \ar[crossing over, hookleftarrow, swap]{ll}[near end]{j}\\
&G_- \ar{ld}[swap]{\pi_Z'} \ar[back line, hook]{rr}[near start]{k} 
  & & \shZ_- \ar{ld}[near start]{\pi'}  \\
Z  \ar[hook]{rr}{i} & & X  \ar[crossing over, leftarrow]{uu}[near end, swap]{\pi}
	\end{tikzcd}
	\end{equation}
where  $\Gamma_Z : = Z \times_X \Gamma =G_+ \times_Z G_-$. Then the normal bundles are given by 
	\begin{align*}
	& \sN_i = G_Z \otimes K_Z ,  &\sN_j =  \shQ_+^\vee \boxtimes K_Z, \\
	&\sN_k =G_Z \boxtimes \shQ_-^\vee,  & \sN_\ell = \shQ_+^\vee  \boxtimes \shQ_-^\vee. 
	\end{align*}
The excess bundle for the front square is given by $\sV = \shU_+^\vee \boxtimes K_Z$, and the excess bundle for the back square is $\sV' =  \shU_+^\vee \boxtimes \shQ_-^\vee$. Therefore 
	\begin{equation*}
		\pi^* \, i_* (\blank) = j_*(c_{\rm top}(\sV) \cap \pi_Z^*(\blank)), \qquad r_{-}^* k_* (\blank) = \ell_* (c_{\rm top}(\sV') \cap r_{Z-}^*(\blank)).
	\end{equation*}
Similarly the excess bundle for the bottom square is given by $\sW = G_Z \boxtimes \shU_-^\vee$, and for the top square is $\sW' =\shQ_+^\vee \boxtimes \shU_-^\vee$. Therefore 
	\begin{equation*}
		\pi'^* \, i_* (\blank) = k_*(c_{\rm top}(\sW) \cap \pi_Z'^*(\blank)), \qquad r_{+}^* j_* (\blank) = \ell_* (c_{\rm top}(\sW') \cap r_{Z+}^*(\blank)).
	\end{equation*}
\end{lemma}

\begin{proof} As the statements are local, it suffices to assume that there are vector bundles $E, F$ over $X$ of rank $n,m$, a morphism $\sigma \colon F \to E$ such that $\sG \simeq \Coker \sigma$ and  $n - m = \rank \sG$. Then over $Z$, the kernel and cokernel of $\sigma$ are locally free sheaves by Lem. \ref{lem:deg:normal}. Thus there is an exact sequence of vector bundles over $Z$:
	$$0 \to K_Z^\vee \to F|_Z \to E|_Z \to  G_Z \to 0.$$ 
Since $\Coker (K_Z^\vee \to F_Z) = \im (\sigma|_Z)=: B_Z$ has constant rank, therefore it is a vector subbundle of $E|_Z$, and above sequence splits into two short exact sequences of vector bundles $0 \to K_Z^\vee \to F|_Z \to B_Z \to 0$ and $0 \to B_Z \to E|_Z \to  G_Z \to 0$. Then the normal bundle $\sN_i = G_Z \otimes K_Z$ follows from Lem. \ref{lem:deg:normal}. 

We next show the other normal bundle also fits into this framework. For example, consider the embedding $j \colon G_+ \hookrightarrow \shZ_+$. Notice there is a natural embedding $\shZ_+ \hookrightarrow \Gr_{d_+}(E^\vee)$ induced by $E \twoheadrightarrow \sG$. Denote the tautological sequence on $\Gr_{d_+}(E^\vee)$ by $0 \to \shU(E^\vee) \to E^\vee \to \shQ(E^\vee) \to 0$. Then $\shZ_+$ is precisely the loci where the composition $F \to E \to \shU^\vee(E^\vee)$ is zero, therefore over $\widetilde{X}: = \shZ_+ \subset \Gr_{d_+}(E^\vee)$, $\pi^*\sigma \colon \pi^* F \to \pi^* E$ factors through a map between vector bundles 
	$$\widetilde{\sigma} \colon \widetilde{F} \to  \widetilde{E}, \quad \text{where} \quad \widetilde{F}:= \pi^* F,  \quad \widetilde{E}:= \shQ^\vee(E^\vee).$$
Notice by our assumption $d_+ \le r+ i + 1$, i.e. $\rank B = m-(i+1) \le n - d_+ =\rank \shQ^\vee(E^\vee)$, therefore over every closed point $x \in X$ the rank of $\sigma$ agrees with the rank of $\widetilde{\sigma}$. Therefore exactly over $\widetilde{Z}: = \pi^{-1}(Z) \simeq G_+ \subset \widetilde{X}$, the rank of $\widetilde{\sigma}$ achieves its minimal. Hence $\widetilde{\sigma} \colon \widetilde{F} \to  \widetilde{E}$ has constant (minimal) rank $m-(i+1)$ over $\widetilde{\shZ} \subset \widetilde{X}$.

Notice there is a natural embedding $i_{G_+}: G_+ \hookrightarrow \Gr_{d_+}(E|_Z^\vee)$ induced by the surjection $E|_Z \twoheadrightarrow G_Z$. If we denote the tautological sequences on $\Gr_{d_+}(E|_Z^\vee)$ by 
	$0 \to \shU(E|_Z^\vee) \to E|_Z^\vee \to  \shQ(E|_Z^\vee) \to 0,$
then $\shU(E|_Z^\vee) = \shU(E^\vee)|_Z$, $\shQ(E|_Z^\vee) = \shQ(E^\vee)|_Z$. Dualizing above sequence,
	$$0 \to \shQ^\vee(E|_Z^\vee)  \to E|_Z \to \shU^\vee(E|_Z^\vee) \to 0,$$
 and comparing with the (dual) tautological sequence on $G_+ = \Gr_{d_+}(G_Z)$:
	$$0 \to \shQ_+^\vee \to G_Z \to \shU_+^\vee \to 0,$$ 
since $E|_Z \twoheadrightarrow G_Z$ induces isomorphism of quotients $i_{G_+}^* \shU^\vee(E|_Z^\vee) \simeq \shU_+^\vee$. Hence over $\widetilde{Z}$ there is a commutative diagram:
	\begin{equation*}
	\begin{tikzcd}
		\pi_Z^* B_Z  \ar[dashed]{r} \ar[equal]{d} & \shQ^\vee(E|_Z^\vee) |_{\widetilde{Z}} \ar{d} \ar[dashed]{r} & \shQ_+^\vee\ar{d}  \\
		\pi_Z^* B_Z  \ar{r} & \pi_Z^* E_Z \ar{r} \ar{d} & \pi_Z^* G_Z \ar{d} \\
			& i_{G_+}^* \shU^\vee(E|_Z^\vee)  \ar{r}{\sim}  &  \shU_+^\vee,
	\end{tikzcd}
	\end{equation*}
where the three columns and the last two rows are exact, hence it induces a short exact sequence on the first row. Hence there is an exact sequence of vector bundles on $G_+$:
		$$0 \to \pi_Z^* K_Z^\vee \to \widetilde{F}|_{ \widetilde{Z}}  \xrightarrow{\widetilde{\sigma}|_{\widetilde{Z}}} \widetilde{E}|_{ \widetilde{Z}} \to \shQ_+^\vee \to 0,$$
where the middle map $\widetilde{\sigma}|_{\widetilde{Z}}$ factors through $\widetilde{F}|_{ \widetilde{Z}}  \twoheadrightarrow \pi_Z^* B \hookrightarrow  \widetilde{E}|_{ \widetilde{Z}}$.
 	Therefore from Lem. \ref{lem:deg:normal}, normal bundle $\sN_j \simeq \shQ_+^\vee \boxtimes K_Z$ and the excess bundle 
	$\sV = \pi_Z^* \sN_i/ \sN_j \simeq \shU_+^\vee \boxtimes K_Z.$
	
The rest of statements follow from a similar argument:  for example, the statement about $\sN_\ell$ and the corresponding excess bundle follows from consider the map $\widetilde{\sigma}^\vee \colon \widetilde{E}^\vee \to  \widetilde{F}^\vee$ over $\widetilde{Z}= G_+ \subseteq \widetilde{X} = \shZ_+$; and the rest two situations  follow from replacing $+$ by $-$.
\end{proof}

\begin{lemma}\label{lem:commute} In the situation of Lem. \ref{lem:stratum}, for any two Young diagrams $\nu, \mu$ inside a fixed box $B$ of $\le d_+$ entires (for example $B=B_{d_+ - d_-, \ell_+ - \ell_-}$), consider the following maps
	\begin{align*}
	& \Psi^\nu (\blank): = r_{Z+*}(c_{\rm top}(\sV') \cup \Delta_{\nu}(-\shU_+) \cap  r_{Z-}^*(\blank))   \colon & CH(G_-) \to CH(G_+); \\
	& \Psi^{\shU}_{\mu}(\blank): = r_{Z-*}(c_{\rm top}(\sW') \cup \Delta_{\mu^c}(-\shU_+) \cap r_{Z+}^*(\blank))  \colon & CH(G_+) \to CH(G_-); \\
	& \Gamma^\nu (\blank): = r_{+*}(\Delta_{\nu}(-\shU_+(\sG)) \cap  r_{-}^*(\blank))   \colon & CH(\shZ_-) \to CH(\shZ_+);\\
	& \Gamma^{\shU}_{\mu}(\blank): = r_{-*}(\Delta_{\mu^c}(-\shU_+(\sG)) \cap r_{+}^*(\blank))  \colon & CH(\shZ_+) \to CH(\shZ_-).
	\end{align*}
Then the following holds:
		$$\Gamma^\nu k_* (\blank) = j_* \Psi^\nu (\blank), \qquad \Gamma^{\shU}_{\mu} j_* (\blank) = k_* \Psi^{\shU}_{\mu}(\blank).$$
\end{lemma}
Notice the above map $\Psi^{\nu}$ is exactly the same map $\Psi^{\nu}$ defined in Lem. \ref{lem:virtual:flip}.
\begin{proof} From excess bundle formula (see \cite[Thm. 6.3]{Ful} and  \cite[Prop. 6.2(1), Prop. 6.6]{Ful}),
\begin{align*}
	&\Gamma^\nu k_* (\blank) = r_{+*} (\Delta_{\mu}(-\shU_+(\sG)) \cap r_-^*k_* (\blank))  = r_{+*} \ell_{*} (\ell^*\Delta_{\mu}(-\shU_+(\sG)))  \cap c_{\rm top}(\sV') \cap r_{Z-}^* (\blank) )  \\
	&= j_* r_{Z+\,*}  (c_{\rm top}(\sV') \cup \Delta_{\mu}(-\shU_+) \cap r_{Z-}^* (\blank) ) = j_* \Psi^\nu (\blank),
	\end{align*}
and similarly for the other identity.
\end{proof}

We are now ready to prove the main theorem \ref{thm:main}. 
\begin{proof}[Proof of Thm. \ref{thm:main}]
Denote $X_i:= X^{\ge \delta+i+1}(\sG)$ for $i \ge -1$, then there is a stratification $\ldots \subset X_{i+1} \subset X_{i} \subset \ldots \subset X_{1} \subset X_{0} \subset X_{-1} = X$. Then this stratification induces stratifications of $\shZ_d^+$, $\shZ_{d-j}^-$, $j \in [0,d]$ and $\Gamma_{(d-j,d)}$ through pullbacks from $X$, see diagram (\ref{diagram:Gamma}). 

For simplicity, for each pair $(i,\ell)$ with $\ell  \ge  i \ge -1$, denote by $X_{i\backslash \ell} : = X_{i} \backslash X_{\ell}$, $(\shZ_d^+)_{i\backslash \ell}:= (\shZ_d^+)_i \backslash (\shZ_d^+)_{\ell}$, $(\shZ_{d-j}^-)_{i\backslash \ell}: =(\shZ_{d-j}^-)_i \backslash (\shZ_{d-j}^-)_{\ell}$ and $(\Gamma_{(d-j,d)})_{i\backslash \ell} : =(\Gamma_{(d-j,d)})_i \backslash (\Gamma_{(d-j,d)})_{\ell}$. 
For each $i \ge -1$, we will denote the natural inclusions by: $i_{i} \colon X_i \hookrightarrow X$, $j_{i} \colon (\shZ_d^+)_i \hookrightarrow \shZ_d^+$,  $k_{(j), i} \colon (\shZ_{d-j}^-)_{i} \hookrightarrow \shZ_{d-j}^-$ and $\ell_{(j), i} \colon (\Gamma_{(d-j,d)})_i \hookrightarrow \Gamma_{(d-j,d)}$. For $i \ge 0$, we also denote $i_{i,i-1} \colon X_{i} \hookrightarrow X_{i-1}$ the natural inclusion, and $j_{i,i-1}$, $k_{(j), i,i-1}$ and $\ell_{(j), i,i-1}$ are defined similarly.

Then over each stratum $X_{i \backslash i+1} = X_{i} \backslash X_{i+1}$, $\sG$ (resp. $\sK$) has constant $\delta+i+1$ (resp. $i+1$), and the diagram (\ref{diagram:Gamma}) for all $j$ becomes: 
\begin{equation*}
		\begin{tikzcd}[row sep= 3 em, column sep = 8 em]
	\bigsqcup_{j=0}^{\delta} (\Gamma_{(d-j,d)})_{i \backslash i+1}  \ar{d}[swap]{\Gr_d(\delta+i+1)\text{-bundle}} \ar{r}{\bigsqcup_{j=0}^{\delta}\Gr_{d-j}(i+1)\text{-bundle}} & (\shZ_d^+)_{i \backslash i+1}  \ar{d}{\Gr_d(\delta+i+1) \text{-bundle}} 
	\\
	\bigsqcup_{j=0}^{\delta} (\shZ_{d-j}^-)_{i \backslash i+1}   \ar{r}{\bigsqcup_{j=0}^{\delta} \Gr_{d-j}(i+1)\text{-bundle}}         &X_{i \backslash i+1}
		\end{tikzcd}		
\end{equation*}
For any fixed integer $i \ge 0$, under the condition of Thm. \ref{thm:main}, then $Z: = X_{i\backslash i+1} \subset X \backslash X_{i+1} = X_{-1 \backslash i+1}$ is a locally complete intersection subscheme of codimension $(i+1)(\delta+i+1)$, and $\sG$ has constant rank $\delta+i+1$ over $Z$. Therefore the conditions of Lem. \ref{lem:stratum} \& \ref{lem:commute} are satisfied by $Z \subset X \backslash X_{i+1}$ and $\sG$. The rest of the proof follows a very similar strategy as \cite{J19}.

\medskip \textit{Surjectivity of the map (\ref{eqn:main.thm}).} Similar to \cite{J19}, for each $i \ge -1$, $\exists$ exact sequence:
	$$
	\begin{tikzcd}
	 CH((\shZ_d^+)_{i \backslash i+1})  \ar{r}{j_{i\,*} }& CH(\shZ_d^+ \backslash (\shZ_d^+)_{ i+1})    \ar{r} &  CH(\shZ_d^+\backslash (\shZ_d^+)_{i})  \ar{r} & 0.
	\end{tikzcd}
	$$
Therefore by induction we see $CH(\shZ_d^+)$ is generated by the images of $j_{i \,*} \colon CH((\shZ_d^+)_{i \backslash i+1}) \to CH(\shZ_d^+)$ for all strata $(\shZ_d^+)_{i \backslash i+1}$, $i \ge -1$, where $i=-1$ corresponds to the open stratum. 

Hence we need only show that the image of the map (\ref{eqn:main.thm}) contains the image of all $j_{i \,*} \colon CH((\shZ_d^+)_{i \backslash i+1}) \to CH(\shZ_d^+)$. The top stratum $i=-1$ follows from Grassmannian bundle formula. If $i \ge 0$, set $Z: = X_{i\backslash i+1} \subset X \backslash X_{i+1}$ as above. For simplicity of notations we omit the subindex $i$ and denote $j_{*} :=j_{i \,*}$, $k_{*}: =k_{i \,*}$, etc. Therefore we are in this situation of Lem. \ref{lem:stratum} \& \ref{lem:commute}, which reduces to the virtual Grassmann flip case \S \ref{sec:virtual.flips}. 
In particular, Thm. \ref{thm:main:local} implies that for any $\gamma \in CH((\shZ_d^+)_{i \backslash i+1})$, there exists $\gamma_{(j, \nu^{(j)})} \in CH(\shZ_{d-j}^{-})$ such that 
	$$\gamma = \sum_{j=0}^\delta \sum_{\nu^{(j)} \in B_{j,\delta-j}} \Psi_{(d-j,d)}^{\nu^{(j)}} (\gamma_{(j, \nu^{(j)})} ).$$ 
However from Lem. \ref{lem:commute},
	$j_* \, \Psi_{(d-i,d)}^{\nu^{(i)}} = \Gamma_{(d-i,d)}^{\nu^{(i)}} \, k_{(j) \,*}.$
Therefore 
	$$j_* \gamma = \sum_{j=0}^\delta \sum_{\nu^{(j)} \in B_{j,\delta-j}} \Gamma_{(d-i,d)}^{\nu^{(i)}} \left( k_{(j)\,*} (\gamma_{(j, \nu^{(j)})})) \right) = \Gamma^* \left(\bigoplus_{(j,\nu^{(j)})} k_{(j)\,*} (\gamma_{(j, \nu^{(j)})}) \right).$$
This shows the surjectivity of the map (\ref{eqn:main.thm}).

\medskip \textit{Injectivity of the map (\ref{eqn:main.thm}).}  For each $i \ge -1$, we denote $k_{i\,*}$ the following map:
	$$k_{i\,*}: = \bigoplus_{j=0}^{\delta} k_{(j), i\,*}  \colon \bigoplus_{j=0}^{\delta} CH((\shZ_{d-j}^-)_{i \backslash i+1})  \to \bigoplus_{j=0}^{\delta} CH((\shZ_{d-j}^-)_{-1 \backslash i+1}).$$
Then there is a commutative diagram of {\em short exact sequences}:
	\begin{equation} \label{eqn:ses}
	\begin{tikzcd}[row sep=2 em, column sep=1.6 em]
	0  \ar{r} & \Im k_{i\,*}  \ar{r} \ar{d}{\Gamma^*|_{\Im k_{0\,*}}}  &  \bigoplus_{j=0}^{\delta} CH((\shZ_{d-j}^-)_{-1 \backslash i+1})  \ar{r}  \ar{d}{\Gamma^*|_{-1 \backslash i+1}} &   \bigoplus_{j=0}^{\delta} CH((\shZ_{d-j}^-)_{-1 \backslash i})  \ar{r} \ar{d}{\Gamma^*|_{-1 \backslash i}} & 0. \\
	0   \ar{r} &  \Im j_{i\,*}   \ar{r} & CH((\shZ_d^+)_{-1 \backslash i+1})  \ar{r} &  CH((\shZ_d^+)_{-1 \backslash i})  \ar{r} & 0.
	\end{tikzcd}
	\end{equation}
If we assume for each $i \ge 0$ the map $\Gamma^*|_{\Im k_{i\,*}}$ is injective, then the injectivity of map (\ref{eqn:main.thm}) follows from induction. In fact, we can show by induction that the map of middle column is injective for all $i$. For the base case, there are two possibilities. If $d \le \delta$, then $\shZ_0^-= X$, $\shZ_1^- = \PP(\sK)$ are the only two schemes of $\{\shZ_{d-j}^-\}$ supported over $X_{1}$, and (\ref{eqn:ses}) becomes:
	\begin{equation*}
	\begin{tikzcd}[row sep=2 em, column sep=1.6 em]
	0  \ar{r} & \Im k_{0\,*} \ar{r} \ar[hook]{d}{\Gamma^*|_{\Im k_{i\,*}}}  & CH(\PP(\sK)_{-1 \backslash 1}) \oplus CH(X_{-1 \backslash 1}) \ar{r}  \ar{d}{\Gamma^*|_{-1 \backslash 1}} & CH(X)  \ar{r} \ar[hook]{d}{\Gamma^*} & 0. \\
	0   \ar{r} & \Im j_{0\,*}   \ar{r} & CH((\shZ_d^+)_{-1 \backslash 1})  \ar{r} &  CH((\shZ_d^+)_{-1 \backslash 0})  \ar{r} & 0,
	\end{tikzcd}
	\end{equation*}
The injectivity of last column follows from Grassmannian bundle formula. Therefore the middle map is injective for the base case $i_{\rm min}=0$. If $d > \delta$, then $\shZ_{d-\delta}^{-}$ is supported over $X_{d-\delta-1}$, and the base case is $i_{\rm min} = d - \delta$, where (\ref{eqn:ses}) becomes:
	\begin{equation*}
	\begin{tikzcd}[row sep=2 em, column sep=1.6 em]
	0  \ar{r} & \Im k_{i_{\rm min}\,*} \ar[equal]{r} \ar[hook]{d}{\Gamma^*|_{\Im k_{i_{\rm min}\,*}}}  & CH((\shZ_{d-\delta}^-)_{-1 \backslash i_{\rm min}+1})  \ar{r}  \ar{d}{\Gamma^*|_{-1 \backslash i_{\rm min}+1}} & 0  \ar{r} \ar{d}{0} & 0. \\
	0   \ar{r} & \Im j_{i_{\rm min}\,*}   \ar[hook]{r} & CH((\shZ_{d}^+)_{-1 \backslash i_{\rm min}+1})  \ar{r} &  CH((\shZ_{d}^+)_{-1 \backslash i_{\rm min}})  \ar{r} & 0.
	\end{tikzcd}
	\end{equation*}
Therefore the middle map is always injective. The induction step follows directly from diagram (\ref{eqn:ses}). Hence the middle map is injective for all $i$, in particular for $i=i_{\rm max}$  \footnote{i.e. $i_{\rm max}$ is the minimal positive number such that $X_{i_{\rm max}+1} = \emptyset$; since $X$ is locally Noetherian of pure dimension, $i_{\rm max}$ always exists.}. Then $\Gamma^* = \Gamma^*|_{-1 \backslash i_{\max}+1}$ is injective on the whole space, the theorem is proved.

\medskip
It remains to prove for each $i \ge 0$ the map $\Gamma^*|_{\Im k_{i\,*}}$ is injective. Set $Z: = X_{i\backslash i+1} \subset X \backslash X_{i+1}$ as above, and for simplicity of notations we will omit the subindex $i$ of all notations in the rest of the proof. The goal is to show that for any $\displaystyle \gamma= \bigoplus_{(j, \nu^{(j)})} \gamma_{(j, \nu^{(j)})} \in \bigoplus_{(j, \nu^{(j)})} CH(\shZ_{d-j}^{-})$,
	$$\Gamma^* \,  k_* \, \gamma =  j_* \, \Psi^* \, \gamma =0 \quad \implies \quad k_* \, \gamma=0.$$
Here $\Psi^*$ denotes the map of isomorphism of Thm. \ref{thm:main:local} (2), and the commutativity $\Gamma^* \,  k_* = j_* \, \Psi^*$ is from Lem. \ref{lem:commute}. Similar to  proof of Thm. \ref{thm:main:local}, we show by (inverse) induction that $k_* \, \gamma_{(j, \nu^{(j)})} = 0$ for all indices $(j, \nu^{(j)})$. Recall the partial order $\prec$ defined by (\ref{eqn:porder}).

\begin{itemize}[leftmargin = *]
	\item {\em Base case.} Start with the maximal index. There are two possibilities. If $\delta \ge d$, then the maximal element is $(j_{\rm max},  \nu^{(j)}_{\rm max}) = (\delta, (0))$, then from Thm. \ref{thm:main:local},
		$$\gamma_{(\delta, (0))} = \Psi_{(d-\delta,d), (0)} (\Psi^* (\gamma)), \quad \text{where} \quad \Psi_{(d-\delta,d), (0)} =  r_{Z-*}(c_{\rm top}(\sW')  \cap r_{Z+}^*(\blank)),$$
where for the last expressions the notations are from Lem. \ref{lem:stratum}. But from Lem. \ref{lem:commute} (notice $\Psi_{(d-\delta,d), (0)} = \Psi_{(d-\delta,d), (0)}^\shU$ in this case),  $k_* \, \Psi_{(d-\delta,d), (0)} =  \Gamma_{(d-\delta, d), (0)}^{\shU} , j_*$, therefore
		$$
		k_* \,\gamma_{(\delta, (0))}  = k_* \, \Psi_{(d-\delta,d), (0)}^{\shU} (\Psi^* (\gamma)) = \Gamma_{(d-\delta, d), (0)}^{\shU} \circ j_* \, \Psi^* (\gamma) = 0. $$ 
	If $d \le \delta$, then the maximal index is $(j_{\rm max}, \nu_{\rm max}^{(j_{\rm max})}) = (d, ((\delta-d)^d))$. From Thm. \ref{thm:main:local},
		$$\gamma_{(d, ((\delta-d)^d))} = \Psi_{(0,d), ((\delta-d)^d)} (\Psi^* (\gamma)),$$
	then similarly Lem. \ref{lem:commute} (as $\Psi_{(0,d), ((\delta-d)^d)} =  \Psi_{(0,d), ((\delta-d)^d)}^\shU$ holds) implies: 
		$$ k_* \,\gamma_{(d, ((\delta-d)^d))}  = k_* \,\Psi_{(0,d), ((\delta-d)^d)} (\Psi^* (\gamma))  = \Gamma_{(0,d), ((\delta-d)^d)}^{\shU} \circ j_* \, \Psi^* (\gamma) = 0. $$
	\item {\em Induction step.} If $k_* \,\gamma_{(\ell, \tau^{(\ell)})}  =0 $ for all $(\ell, \tau^{(\ell)}) \succ (j, \nu^{(j)})$, if we replace $\gamma$ by: 
	$$\gamma_{\rm new} := \gamma - \sum_{(\ell, \tau^{(\ell)}) \succ (j, \nu^{(j)})} \gamma_{(\ell, \tau^{(\ell)})},$$ 
	then $k_* \,\gamma_{\rm new}  = k_* \, \gamma$. Hence we may assume $\gamma_{(\ell, \tau^{(\ell)})}  =0$ for all $(\ell, \tau^{(\ell)}) \succ (j, \nu^{(j)})$. Now
	$$ \gamma_{(j, \nu^{(j)})} =(-1)^{(d-j)(d-\delta)} \cdot \Psi^{\rm std}_{(d-j,d), \,\nu^{(j)}} \circ \Psi^* \, \gamma,$$
by semi-orthogonality of Thm. \ref{thm:main:local} (1). As $-\shU_+ = \shQ_+ - G_Z^\vee \in K_0(Z)$, hence by Lem. \ref{lem:Gr:change_basis}
	\begin{align*}
	\Psi^{\shU}_{(d-j,j), \, \nu^{(j)}} (\Psi^* \, \gamma) = \Psi^{\rm std}_{(d-j,d), \,\nu^{(j)}} ( \Psi^* \, \gamma) + \sum_{\mu \supsetneq \nu^{(j)}} \Delta_{\nu^{(j), c}/\mu^c}(-G_Z^\vee) \cdot  \Psi^{\rm std}_{(d-j,d), \,\mu} ( \Psi^* \, \gamma).
	\end{align*}
	But by our assumption from the induction hypothesis,  
		$$\displaystyle \Psi^* \gamma \in \Im \left( \bigoplus_{(t, \tau^{(t)}) \nsucc (j, \nu^{(j)})} \Psi_{(d-t,d)}^{\tau^{(t)}} \right),$$
where the indices satisfy $(t, \tau^{(t)}) \nsucceq (j, \mu)$ for any $\mu \supsetneq \nu^{(j)}$. Therefore by semiorthogonality of Thm. \ref{thm:main:local} (1), $ \Psi^{\rm std}_{(d-j,d), \,\mu} ( \Psi^* \, \gamma) =0$.  
	Hence 
	$$\Psi^{\shU}_{(d-j,d), \, \nu^{(j)}} (\Psi^* \, \gamma)  =\Psi^{\rm std}_{(d-j,d), \,\nu^{(j)}} ( \Psi^* \, \gamma) = \pm  \gamma_{(j, \nu^{(j)})}.$$
	On the other hand, by Lem. \ref{lem:commute} 
		$$k_* \, \gamma_{(j, \nu^{(j)})} = \pm \,k_* \, \Psi^{\shU}_{(d-j,d), \, \nu^{(j)}} (\Psi^* \, \gamma) =  \pm \, \Gamma_{(d-j,d),\, \nu^{(j)}}^{\shU} \, j_* (\Psi^* \, \gamma) = 0.$$
	Hence the induction step is proved.
\end{itemize}
By induction, the claim is proved. This concludes the proof of theorem \ref{thm:main}. \end{proof}


\section{Applications}
\subsection{Blowup of determinantal subschemes}


\begin{lemma}\label{lem:Quot=Bl} Let $X$ be any irreducible $\kk$-scheme, $\sigma \colon F \to E$ be a  $\sO_X$-module morphism between vector bundles which is generically injective, $\rank F = m$, $\rank E = m + \delta$, $\delta \ge 1$, and denote $\sG: = \Coker \sigma$ the cokernel. Let $Z: = X^{\ge \delta+1}(\sG) \subset X$ be the determinantal subscheme, i.e. its ideal $\sI_{Z}$ is generated by the $m \times m$ minors of $\sigma$. Assume that:
	$$\codim (X^{\ge\delta +i}(\sG) \subset X) \ge i\cdot \delta +1 \qquad \forall i \ge 1.$$
Then $\foQuot_{\delta}(\sG)$ is irreducible and isomorphic to $\Bl_{Z} X$, the blowup of $X$ along $Z \subset X$.
\end{lemma}
This type of results are well-known, for example the case $\delta=1$ could be found in \cite{ES}. 
We present a proof in our setting for the sake of completeness.
\begin{proof} Let $q \colon \Gr_\delta(E^\vee) \to X$ be the Grassmannian bundle over $X$, with tautological sequence 
		$$0 \to \shU \to q^*E^\vee \to \shQ \to 0,$$
where $\shU$ is the universal subbundle of rank $\delta$, and $\shQ$ is the universal quotient bundle of rank $m$. As $\bigwedge^m \shQ$ is relative ample over $X$, hence it defines a relative Pl\"ucker embedding:
		$$\Gr_\delta(E^\vee) = \Gr(E^\vee,m) \hookrightarrow \PP(\bigwedge^m q_* \shQ) \subset \PP(\bigwedge^m E^\vee).$$
Denote $\pi \colon \foQuot_\delta(\sG) \to X$ the projection. Same as the proof of Lem. \ref{lem:stratum}, the surjection $E \twoheadrightarrow \sG$ induces an inclusion $\foQuot_\delta(\sG) \subset \Gr_\delta(E^\vee)$, and $\foQuot_\delta(\sG)$ is identified with the loci where the composition $q^* F \to q^* E \to \shU^\vee$ is zero, or equivalently, where the map $q^* \sigma$ factors through $\widetilde{\sigma} \colon \pi^*F \to \shQ^\vee$. As $\widetilde{\sigma}$ is generically injective, hence $\pi^{-1} \sI_Z \subset \sO_{\foQuot_\delta(\sG)}$ is invertible, generated by $\bigwedge^m \widetilde{\sigma} \in \pi^*  \det F^\vee \otimes \det \shQ^\vee$. As $\foQuot_{\delta}(\sG)$ is a $\Gr_{\delta}(\delta+i)$-bundle over $X^{\ge \delta +i}(\sG) \backslash X^{\ge \delta +i+1}(\sG)$, therefore the dimension condition of the lemma implies that 
	$$\dim \pi^{-1}(X^{\ge \delta +i}(\sG) \backslash X^{\ge \delta +i+1}(\sG)) \le \dim X - i\delta - 1 + i \delta \le \dim X - 1,$$
for all $i \ge 1$. (Note $\codim (Z \subset X) = \delta +1$, thus the isomorphism $\pi \colon \foQuot_{\delta}(\sG)|_{X\backslash Z} \simeq X\backslash Z$ is over an open dense subset $X\backslash Z$.) Hence $\foQuot_{\delta}(\sG)$ is irreducible.

By the definition of blowup, it remains to show that for any $f \colon Y \to X$ such that $f^{-1} \sI_Z \cdot \sO_Y$ is invertible, then $f$ factors uniquely through $\pi$. Notice $f^{-1} \sI_Z $ is generated by the minors of the map $f^{*}\sigma \colon f^* F \to f^* E$, hence by definition there is a surjection
	$$f^* (\det F \otimes \bigwedge^m N^\vee) \twoheadrightarrow f^{-1}\sI_Z \subseteq \sO_Y,$$
thus a surjection $f^* (\bigwedge^m N^\vee) \twoheadrightarrow f^*\det F^\vee \otimes f^{-1}\sI_Z $. As $f^{-1}\sI_Z$ is invertible, by Grothendieck's characterisation of projectization Example \ref{ex:proj}, this defines a map 
	$$\phi \colon Y \to  \PP(\bigwedge^m E^\vee)$$
lifting $f$. (Notice that this also implies $f^* \sigma$ is generically injective.) Over the open dense subset $Y \backslash  f^{-1}(Z)$, $f^* \sigma \colon f^* M \hookrightarrow f^* N$ defines a vector subbundle, therefore $\phi|_{Y \backslash  f^{-1}(Z)}$ factors through $\Gr_\delta(E^\vee) \subset \PP(\bigwedge^m E^\vee)$ and an injection into $\foQuot_\delta(\sG)|_{X \backslash Z}$. Since $\foQuot_\delta(\sG)$ is proper over $X$ and irreducible,  $\phi|_{Y \backslash  f^{-1}(Z)}$ extends to a unique map $Y \to \foQuot_\delta(\sG)$.
\end{proof}	

\begin{theorem}[Blowup of determinantal subscheme] \label{thm:Bl:det} Under the same condition of Thm. \ref{thm:main} and assume $X$ is irreducible. Then the top degeneracy loci $Z := X^{\ge \delta +1}(\sG) \subset X$ is a Cohen--Macaulay determinantal subscheme of codimension $\delta +1$, and admits a stratification:
	$$ \cdots \subset Z_{i+1} \subset Z_{i} \subset \cdots \subset Z_{1} \subset Z_{0} = Z, \qquad Z_i := X^{\ge \delta + i +1}(\sG), i \ge 0.$$
Then 
	$$\widetilde{Z_i} : = \foQuot_{X, i+1}(\sExt^1_{X}(\sG,\sO_X)) \to Z_i, \qquad i \ge 0$$
is an IH-small desingularization of $Z_{i}$. Denote $\Bl_{Z} X$ the blowing up of $X$ along $Z$ and $\widetilde{Z} : = \widetilde{Z_0}$, then for any $k \ge 0$, there is an isomorphism of Chow groups:
	\begin{align*}
	CH_k(\Bl_{Z} X) \simeq CH_k(X) \oplus \bigoplus_{\ell =0}^{\delta-1} CH_{k-\delta+\ell }(\widetilde{Z}) \oplus \bigoplus_{i =1}^{\delta-1} \bigoplus_{\nu^{(i)} \in B_{\delta-(i+1), i+1}}CH_{k- (i+1) \cdot \delta+|\nu^{(i)}|}(\widetilde{Z_i}),
	\end{align*}
where the map is induced by fiber products as Thm. \ref{thm:main}. Furthermore if all schemes involved are smooth and projective over $\kk$, then the same map induces a decomposition of integral covariant Chow motives:
	\begin{align*}
	\foh(\Bl_{Z} X) \simeq \foh(X) \oplus \bigoplus_{\ell =0}^{\delta-1} \foh(\widetilde{Z})(\delta-\ell)  \oplus \bigoplus_{i = 1}^{\delta-1} \bigoplus_{\nu^{(i)} \in B_{\delta-(i+1), i+1}} \foh(\widetilde{Z_i}) ((i+1)\cdot \delta -|\nu^{(i)}|).
	\end{align*}
\end{theorem}
Notice above formula of Chow motives plus the IH-small statements implies isomorphisms of Hodge structures on intersection cohomologies (if $\kk \subset \CC$, via Betti realizaiton), as mentioned in the introduction \S \ref{sec:intro:blowup}.

\begin{proof} We only need to show the IH-small statements. Consider $Z_{i+a} \subset Z_i$ for any $i \ge 1$ and $a \ge 1$. The expected dimension condition implies
	$$\codim (Z_{i+a} \subset Z_i) = (i+a+1)(\delta+i+a+1) - (i+1)(\delta+i+1) = a(a + 2(i+2) + \delta).$$
On the other hand, $\widetilde{Z_i}$ has fiber $\Gr_{i+1}(i+a+1)$ over $Z_{i+a}\backslash Z_{i+a+1}$. Therefore 
  	$$\codim (Z_{i+a} \subset Z_i)  - 2 \dim \Gr_{i+1}(i+a+1) = a^2 + a\delta >0.$$
Hence $\widetilde{Z_i} \to Z_i$ is $IH$-small.
\end{proof}
\begin{example}\label{ex:blowup:det}
	\begin{enumerate}
		\item If $\sG = \Coker (\sO_X \xrightarrow{s} E)$ for a regular section $s \in H^0(X,E)$ of a vector bundle $E$, then $\delta = \rank E -1$, $\foQuot_{\rank E - 1}(\sG) \simeq \Bl_Z X$ is the blowup along a locally complete intersection subscheme, and $\widetilde{Z} = Z$, $Z_i = \emptyset$ for $i \ge 1$. The theorem reduces to the usual blowup formula Ex. \ref{ex:cayley.blowup}:
					$$CH_k(\Bl_{Z} X) \simeq CH_k(X) \oplus \bigoplus_{\ell =0}^{\rank E-2} CH_{k-(\rank E-1)+\ell }(Z).$$
		\item If $\rank \sG = \delta = 1$, then $\sG = \sI_Z \otimes \sL$ for some line bundle $\sL$, $Z \subset X$ is codimension $2$, and $\foQuot_{1}(\sG)\simeq \PP(\sI_Z) = \Bl_Z X$ is projectivization. Then the theorem reduces to blowup formula along Cohen--Macaulay codimension $2$ subscheme \cite{J19}:
			$$CH_k(\Bl_{Z} X) \simeq CH_k(X) \oplus CH_{k-1}(\widetilde{Z}).$$
	\end{enumerate}
\end{example}
Therefore above theorem is a generalisation of above known blowup formulae, and shows how higher degeneracy loci contributes to $\Bl_Z X$ when $Z$ is singular and $\codim_X Z \ge 3$.

\subsection{Applications to Brill--Noether theory of curves}
Let $C$ be a smooth projective curve. Recall the following varieties from \cite{ACGH}: for $d,r \in \ZZ$, the {\em Brill--Noether locus} is:
	$$W_{d}^r := W_{d}^r(C) : = \{\sL \mid \dim H^0(C,\sL) \ge r+1\} \subseteq \Pic^d(C),$$
which parametrizes complete linear series of degree $d$ and dimension at least $r$. The scheme 
	$$G_{d}^r :=G_d^r(C) = \{ \text{$g_d^r$'s on $C$} \}$$
parametrizes (not necessarily complete) linear series of degree $d$ and dimension exactly $r$, 
	\begin{align*}
		\Supp (G_d^r)  = \{ (\sL,V_{r+1}) \mid \sL \in \Pic^d(C), L_{r+1} \in \Gr_{r+1}(H^0(C,\sL))\}.
	\end{align*}
If $d \le g-1$, $G_{d}^r = \widetilde{W_{d}^r}$ is the blow-up of $W_{d}^r$ (see \cite[pp. 177]{ACGH}).

The expected dimension of Brill--Noether loci is the {\em Brill--Noether number}, defined as
	$$\rho:=\rho(g,r,d) := g - (r+1)(g-d+r).$$

The following are classical results of Brill--Neother theory, see \cite[IV,V]{ACGH}:

\begin{proposition} Let $C$ be a smooth curve of genus $g$, and assume $r \ge 0$ and $d \ge 1$.
\begin{enumerate}
	\item $G_d^r$ and $W_{d}^r$ are nonempty if $\rho(g,r,d) \ge 0$; 
	$G_d^r$ and $W_{d}^r$ are connected if $\rho(g,r,d) \ge 1$.
	\item For a general curve $C$ (in the sense of Petri), $G_d^r \ne \emptyset$ iff $\rho \ge 0$; same holds for $W_d^r$. If $\rho \ge 0$, then $G_d^r$ is reduced, smooth and of pure dimension $\rho$, and $(W_d^r)_{\rm sing} = W_d^{r+1}$ if $r \ge d-g$. If $\rho \ge 1$, then both $G_d^r$ and  $W_d^r$ are irreducible of dimension $\rho$.
\end{enumerate}
\end{proposition}

The Brill--Neother theory fits naturally into the Quot-formula picture as follows. Following convention of  \cite{Tod}, let $n \ge 0$ be a non-negative integer, and consider the Picard scheme $X: = \Pic^{g-1+n}(C)$. Let $D$ be an effective divisor of large degree on $C$, and $\sL_{\rm univ}$ be the universal line bundle of degree $g-1+n$ on $C \times X$, $\pr_C, \pr_X$ be obvious projections, then
	$$\sV := (\RR \pr_{X *} (\pr_{C}^* \sO(D) \otimes \sL_{\rm univ}))^\vee \quad \text{and} \quad \sW := (\RR \pr_{X *} (\pr_{C}^* \sO_D(D) \otimes \sL_{\rm univ}))^\vee$$
are vector bundles on $X$ of rank $\deg(D) + n$ and resp. $\deg(D)$, and there is a natural map
	$\sigma \colon \sW \to \sV$
between the vector bundles. As in \cite{JL18}, denote 
	$$\sG: = \Coker (\sigma \colon \sW \to \sV) \quad \text{and} \quad \sK: = \Coker (\sigma^\vee \colon \sV^\vee \to \sW^\vee)$$
then $\sG$ has generic rank $n$ and homological dimension $\le 1$, $\sK$ is a torsion sheaf, and $\sK \simeq \sExt^1(\sG,\sO_X)$. For any point $x=[\sL] \in X$, the fibers of $\sG$ and $\sK$ are:
	$$\sG \otimes k(x) = H^0(C,\sL)^\vee \quad \text{and} \quad \sK \otimes k(x)  =H^1(C,\sL)\simeq H^0(C,\sL^\vee \otimes \omega_C)^\vee.$$ 
Therefore the Brill--Neother loci is nothing but degeneracy loci of $\sG$, and the scheme $G_{g-1+n}^r$ of linear series are exactly the Quot-schemes:
		$$W_{g-1+n}^{r}(C) = X^{\ge r+1}(\sG), \qquad G_{g-1+n}^r(C) = \foQuot_{X, r+1}(\sG),  \quad \forall r \in \ZZ.$$
Via the identification $\Pic^{g-1-n}(C)  \simeq  \Pic^{g-1+n}(C) =X$, $\sL \mapsto \sL^\vee \otimes \omega_C$, we also have:
		$$W_{g-1-n}^{r} (C) \simeq  X^{\ge r+1}(\sK), \qquad G_{g-1-n}^{r}(C) \simeq \foQuot_{X, r+1}(\sK),  \quad \forall r \in \ZZ.$$
Therefore the Quot--formula Thm. \ref{thm:main} immediately implies:
		
\begin{theorem}\label{thm:curves} If $C$ is a general curve (in the sense of Petri), then for any $n \ge 0$, $r \ge 0$, there is an isomorphism of Chow groups:
	\begin{align*} 
	\bigoplus_{j=0}^{\min\{n,r+1\}} \bigoplus_{\nu^{(j)} \in B_{j,n-j}} CH_{k - (r+1)(n-j)+|\nu^{(j)}|} (G_{g-1-n}^{r-j}(C))  \xrightarrow{\sim} CH_k(G_{g-1+n}^r(C)),
	\end{align*}
(where the map is given as Thm. \ref{thm:main}) and an isomorphism of Chow motives:
	\begin{align*} 
	\bigoplus_{j=0}^{\min\{n,r+1\}} \bigoplus_{\nu^{(j)} \in B_{j,n-j}} \foh(G_{g-1-n}^{r-j}(C))((r+1)(n-j)-|\nu^{(j)}|)  \xrightarrow{\sim} \foh(G_{g-1+n}^r(C)).
	\end{align*}
\end{theorem}

 \subsection{Applications to (nested) Hilbert schemes of points on surface.} Let $S$ be a smooth algebraic surface, for any $n \ge 0$, denote $\Hilb_n$ the Hilbert scheme of ideals of $S$ of colength $n$, and denote $Z_n \subset X:=\Hilb_n \times S$ the universal subscheme. Then by \cite{ES}, $Z_n$ is Cohen--Macaulay of codimension $2$, and hence $\sI_{Z_n}$ has homological dimension $1$ and $\sExt^1(\sI_{Z_n}, \sO_X) \simeq \omega_{Z_n}$ is the dualizing sheaf of $Z_{n}$.  
For any $d \ge 1$, recall the generlised nested Hilbert scheme defined in the introduction:
 	$$\Hilb_{n,n+d}^{\dagger}(S) : = \{ (I_n \supset I_{n+d}) \mid I_{n}/I_{n+d} \simeq \kappa(p)^{\oplus d} \text{~ for some $p \in S$}\} \subset \Hilb_n \times \Hilb_{n+d}.$$
 Notice that $I_{n}/I_{n+d} \simeq \kappa(p)^{\oplus d}$ (as $\sO_S$-module) is a quite strong condition if $d \ge 2$, which implies that their differences could not contain any curve--linear direction. Consider: 
 	$$\pi_1 :\Hilb_{n,n+d}^{\dagger}(S) \to \Hilb_{n} \times S, \quad \text{and} \quad \pi_2 :\Hilb_{n,n+d}^{\dagger}(S) \to \Hilb_{n+d} \times S,$$
where $\pi_1(I_n, I_{n+d})=(I_n, p = \supp(I_{n}/I_{n+d}))$ and $\pi_2(I_n, I_{n+d}) = (I_{n+d}, p = \supp(I_{n}/I_{n+d}))$. 
 
 \begin{lemma} $\Hilb_{n,n+d}^{\dagger}(S) \simeq \foQuot_{\Hilb_{n} \times S, d}(\sI_{Z_n}) $ and $\Hilb_{n,n+d}^{\dagger}(S) \simeq \foQuot_{\Hilb_{n+d} \times S, d}(\omega_{Z_{n+d}})$.
 \end{lemma}
\begin{proof} Similar to \cite{ES}, for any $(I_n \supset I_{n+d}) \in \Hilb_{n,n+d}^{\dagger}(S)$, denote $\xi_n:=V(I_n) \subset S$, $\xi_{n+d}: = V(I_{n+d}) \subset S$ the corresponding zero subschemes, then there are short exact sequences:
	$$0 \to I_{n+d} \to  I_{n} \to  \kappa(p)^{\oplus d} \to 0, \qquad 0 \to  \kappa(p)^{\oplus d} \to \sO_{\xi_{n+d}} \to \sO_{\xi_n} \to 0.$$
The first one identifies the fiber of $\pi_1$ over the point $(I_{n}, p)$ with $\pi_1^{-1}(I_{n}, p) \simeq  \foQuot_{d}(I_{n}\otimes \kappa(p))$. Dualizing the second exact sequence, we get:
	$$0 \to \omega_{\xi_n} \to  \omega_{\xi_{n+d}} \to \kappa(p)^{\oplus d} \to 0.$$
This gives rise to a map from the fiber of $\pi_2$ over the point $(I_{n+d}, p)$ to the Quot-scheme, i.e. $\pi_2^{-1}(I_{n+d}, p) \to \foQuot_{d}(\omega_{\xi_{n+d}}\otimes \kappa(p))$, which is an isomorphism since dualizing again we see $\xi_n$ can be recovered from the map $\omega_{\xi_{n+d}} \twoheadrightarrow \kappa(p)^{\oplus d}$. 

These isomorphisms can be naturally globalized and made functorial, as the two short exact sequences about $I_n, I_{n+d}$ and $\omega_n, \omega_{n+d}$ can be naturally globalized to the relative case of families of subschemes and points. Hence we are done.
\end{proof}

In a recent note \cite{BCJ} we showed that the Brill--Noether loci $BN_{i,n} : = X^{\ge i+1} (\sI_{Z_n}) \subset \Hilb_n \times S$ are irreducible and all have expected dimensions:
 	$\codim (BN_{i,n} \subset  \Hilb_n \times S) = i(i+1), \forall i \ge 1.$
Therefore the the Quot--formula Thm. \ref{thm:main} implies
 \begin{theorem}\label{thm:Hilb} For any $n,d \ge 1, \forall k$,  the fiber products induce an isomorphism:
 	$$CH_k(\Hilb_{n,n+d}^{\dagger}(S)) \simeq CH_{k-d}(\Hilb_{n-d,n}^{\dagger}(S)) \oplus CH_{k}(\Hilb_{n-d+1,n}^{\dagger}(S)).$$
 \end{theorem}
Notice if $d=1$, this is the formula for usual nested Hilbert scheme \cite{J19}; if $d \ge 2$, then both $\Hilb_{n,n+d}^{\dagger}(S)$ and $\Hilb_{n-d+1,n}^{\dagger}(S)$ are desingularizations of the Brill--Noether locus $BN_{d-1,n}$, and they are related by a flip $\Hilb_{n,n+d}^{\dagger}(S)  \dashrightarrow \Hilb_{n-d+1,n}^{\dagger}(S)$.

\appendix
\section{Projectors for top and lowest strata} In general the projectors for theorem \ref{thm:main} may be hard to compute. But for the contributions from top and lowest strata they can be expressed nicely.  Let $E$ and $F$ be locally free sheaves of rank $n$ and $m$ over a scheme $X$, $\sigma \in \Hom(F,E) = \Gamma(X,F^\vee \otimes X)$ be an injective $\sO_X$-linear map, and denote $\sG : = \Coker (\sigma \colon F \to E)$ the cokernel. Let $d$ be an integer between $1$ and $r  = \rank \sG = n - m$, and denote $\ell = n -d$. Assume $\sigma$ regarded as a section of $F^\vee \otimes E$ is regular, and denote $Z : = Z(\sigma) \subset X$ the zero loci. Similar to generalized Cayley's trick case, consider $G: = \Gr_d(E^\vee)$, let $\shU = \shU(E^\vee)$ (resp. $\shQ = \shQ(E^\vee)$) the universal rank $d$ (resp. $\ell : = n-d$) subbundle resp. (quotient bundle), i.e.  
	$0 \to \shU \to \pi^* E^\vee \to \shQ \to 0.$
Denote $\shZ_+: = \foQuot_d(\sG)$, then there is a natural inclusion $\iota \colon \shZ_+ \hookrightarrow \foQuot_d(E) = \Gr_d (E^\vee)$ induced by $E \twoheadrightarrow \sG$ as before,
assume $\shZ_+$ has expected dimension $\dim G - md$. Notice the restriction of $\shU^\vee$ to $\shZ_+$ is the tautological rank $d$ quotient bundle of $\sG$. The situation is summarised in the following diagram, with names of maps as indicated:
\begin{equation}\label{diagram:dom.strata}
	\begin{tikzcd}[row sep= 2.6 em, column sep = 2.6 em]
	G_Z:=\Gr_d(E^\vee|_Z) \ar{d}[swap]{p} \ar[hook]{r}{j} & \shZ_+ : = \foQuot_{d}(\sG) \ar{d}{\pi} \ar[hook]{r}{\iota} & G:=\Gr_d(E^\vee) \ar{ld}[near start]{q} 
	\\
	Z \ar[hook]{r}{i}         & X  
	\end{tikzcd}	
\end{equation}

\subsection{Top strata}
If $1 \le d \le r = n-m$, then $\shZ_+$ is generically a $\Gr_d(r)$-bundle over $X$. Notice that $r - d = \ell - m$.
\begin{lemma}  \label{lem:vb} For any $\lambda \in B_{d, \ell - m}$, $k \in \ZZ$, we define the following maps as Lem. \ref{lem:Gr:projector}, 
	\begin{align*}
	& \pi_\lambda^*(\blank) := \Delta_{\lambda} (-\shU)  \cap \pi^*(\blank) \colon &CH_{k-d(\ell - m) +|\lambda|}(X) \to CH_{k}(\shZ_+), \\
	& \pi_{\lambda *} (\blank): =  \sum_{\mu \in B_{d,\ell - m}} \Delta_{\mu/\lambda}(\sG^\vee) \cap  \pi_* (\Delta_{\mu^c} (-\shU) \cap (\blank)) \colon & CH_{k}(\shZ_+) \to CH_{k-d(\ell - m) +|\lambda|}(X).
	\end{align*}	 
where $\sG^\vee \equiv E^\vee - F^\vee \in K_0(X)$, i.e. $c(\sG^\vee)  = c(E^\vee)/c(F^\vee)$; and $\mu^c = (\ell - m) - \mu$ is the complement inside the rectangular $B_{d, \ell - m}$. Then the following holds:
	$$\pi_{\lambda *} \, \pi_{\mu}^* = \delta_{\lambda,\mu} \, \Id_{CH(X)}, \qquad \forall \lambda,\mu \in B_{d, \ell - m}.$$
Hence we have the following split-injective map of the ``vector bundle part" of Chow:
	$$\bigoplus_{\lambda \in B_{d,\ell - m}} \pi_\lambda^* = \bigoplus_{\lambda \in B_{d,\ell - m}}\Delta_{\lambda}(-\shU) \cap \pi^*(\blank)   \colon \bigoplus_{\lambda \in B_{d,\ell - m}}  CH_{k-d(\ell - m)+|\lambda|}(X) \hookrightarrow CH_k(\shZ_+). $$
Furthermore the following holds:
	\begin{align} \label{eqn:pi.vs.q} 
	\pi_{\lambda \, *} (\blank)=  \sum_{\nu \in B_{d,\ell - m}} \Delta_{\nu/\lambda}(-F^\vee)  \cap q_{\nu+m\, *} \, \iota_*(\blank),
	\end{align}
where $q_{\lambda \,*}$ is defined as in Lem. \ref{lem:Gr:projector}, i.e. with respect to the basis $\{\Delta_\lambda (-\shU)\}$.
\end{lemma}

\begin{proof} For simplicity denote $\Delta_{\lambda} = \Delta_{\lambda}(-\shU)$ for both the class on $\shZ_+$ and also for $\Gr_d(E^\vee)$, and note that there is no confusion as $\iota^* \Delta_\lambda (-\shU(E^\vee)) = \Delta_\lambda (-\shU)$. For any $\lambda,\mu \in B_{d, \ell - m}$,
\begin{align*}
	\pi_{\lambda *} (\pi_\mu^* \alpha) = \pi_{\lambda *} (\Delta_{\mu} \cdot \pi^* \alpha) 
	= \sum_{\tau \in B_{d,\ell - m}} \Delta_{\tau/\lambda} (\sG^\vee) \cdot q_* (\Delta_{(\ell - m) - \tau} \cdot c_{top}(F^\vee \otimes \shU^\vee) \cdot \Delta_\mu \cdot q^* \alpha) 
\end{align*} 
As $c_{top}(F^\vee \otimes \shU^\vee) = \sum_{\nu \in B_{d,m}} \Delta_\nu(F^\vee) \cdot \Delta_{m- \nu}$, therefore above computation equals to
	\begin{align*}
	 \pi_{\lambda *} (\pi_\mu^* \alpha)= \sum_{\tau \in B_{d,\ell - m}} \sum_{\nu \in B_{d,m}}  \Delta_{\tau/\lambda} (\sG^\vee)  \Delta_{\nu}(F^\vee) \cdot q_* (\Delta_{(\ell - m) -\tau} \cdot \Delta_{m-\nu} \cdot \Delta_{\mu} \cap q^*\alpha).
	\end{align*} 
From Littlewood-Richardson rule Lem. \ref{lem:LR} and BZ-lemma Lem. \ref{lem:BZ},
	$$\Delta_{(\ell - m) -\tau} \cdot \Delta_{m-\nu} = \sum_{\theta \in B_{d,n-d}} c_{(\ell - m)-\tau, \, m - \nu}^{n-d-\theta} \Delta_{n-d - \theta} = \sum_{\theta \in B_{d,n-d}} c_{\tau, \, \nu}^{\theta} \Delta_{n-d - \theta},$$
now by duality for Grassmannian bundle Lem. \ref{lem:Gr:dual}, above reduces to:
	\begin{align*}
	 \pi_{\lambda *} (\pi_\mu^* \alpha) & = \sum_{\tau \in B_{d,\ell - m}} \sum_{\nu \in B_{d,m}} \sum_{\theta \in B_{d,n-d}}   c_{\tau, \,\nu}^{\theta} \cdot \Delta_{\tau/\lambda} (\sG^\vee)  \cdot \Delta_{\nu}(F^\vee) \cdot \Delta_{\mu/\theta}(-E^\vee) \cap \alpha \\
	 & = \sum_{\tau \in B_{d,\ell - m} ,\theta \in B_{d,n-d}}  \Delta_{\mu/\theta}(-E^\vee) ( \sum_{\nu \in B_{d,m}} c_{\tau, \,\nu}^{\theta} \Delta_{\nu}(F^\vee)) \Delta_{\tau/\lambda} (\sG^\vee) \cap \alpha\\
	 & = \sum_{\tau \in B_{d,\ell - m} ,\theta \in B_{d,n-d}}  \Delta_{\mu/\theta}(-E^\vee) \Delta_{\theta/\tau}(F^\vee) \Delta_{\tau/\lambda} (\sG^\vee) \cap \alpha \\
	 & = \Delta_{\mu/\lambda} (-E^\vee+F^\vee+\sG^\vee) \cap \alpha = \delta_{\lambda,\mu}  \cdot \alpha.
	\end{align*} 
Notice that in every step of above summation, the term is zero unless the indices satisfy $\lambda \subseteq \tau \subseteq \theta \subseteq \mu$ and $\nu \subseteq \theta$, by the properties of $c_{\tau, \,\nu}^{\theta}$,$ \Delta_{\tau/\lambda}$ and $\Delta_{\mu/\theta}$. 

For the last statement, if $\iota_* (\sum_{\lambda \in B_{d,\ell - m}} \iota^* \Delta_\lambda \alpha_\lambda) =\sum_{\mu \in B_{d,n-d}} \Delta_\mu \beta_\mu \in CH(G)$, then:
	\begin{align*}
	& \beta_\mu = \sum_{\lambda \in B_{d, \ell-m}} \sum_{\nu \in B_{d,m}} c_{\lambda, m-\nu}^{\mu} \Delta_\nu(F^\vee) \cdot \alpha_\lambda \\
	&=  \sum_{\lambda \in B_{d, \ell-m}} \sum_{\nu \in B_{d,m}} c_{\mu, \nu}^{\lambda+m} \Delta_\nu(F^\vee) \cdot \alpha_\lambda = \sum_{\lambda \in B_{d, \ell-m}}\Delta_{\lambda+m/\mu}(F^\vee) \cdot \alpha_{\lambda}.
	\end{align*}
Hence $\{\alpha_\lambda\}$ and $\{\beta_\mu\}$ are related by: 
	\begin{align*}
	\beta_\mu =  \sum_{\lambda \in B_{d, \ell-m}}\Delta_{\lambda+m/\mu}(F^\vee) \cdot \alpha_{\lambda}, \qquad \alpha_{\lambda} =  \sum_{\nu \in B_{d,\ell - m}} \Delta_{\nu/\lambda}(-F^\vee)  \cdot \beta_{\nu+m}.
	 \end{align*}
Therefore the last equality (\ref{eqn:pi.vs.q}) holds.
Note that one could also directly deduce the first statement from this relation.
\end{proof}

\begin{remark} \label{rmk:lem:vb} We could also consider $\pi^{\shQ \,*}_{\lambda}(\blank): = \Delta_\lambda(\shQ) \cap (\blank) $ for $\lambda \in B_{d, \ell - m}$. Then by change of basis Lem. \ref{lem:Gr:change_basis} for $\shQ = -\shU - E^\vee$,
	$
		\pi^{\shQ \,*}_{\lambda} (\blank)= \sum_{\mu \subseteq \lambda} \Delta_{\lambda/\mu}(E^\vee) \cdot \pi_\mu^* (\blank).
	$
Consider:
	\begin{align*}
		\pi^{\shQ}_{\lambda\,*} (\blank) : = \sum_{\mu \, \mid  \, \lambda \subseteq \mu \subseteq ((\ell-m)^d)} \Delta_{\mu/\lambda}(-E^\vee) \cdot \pi_{\mu\,*} (\blank).
	\end{align*}
Then clearly $\pi^{\shQ}_{\lambda\,*} \pi^{\shQ \,*}_{\mu} = \delta_{\lambda,\mu} \, \Id_{CH(X)}$ holds. Therefore $\bigoplus_{\lambda \in B_{d,\ell-m}} \pi^{\shQ \,*}_{\lambda}$ also induces an embedding of Chow groups as  Lem. \ref{lem:vb}, with the same image. Furthermore,
	\begin{align*}
	\bigoplus_{\lambda \in B_{d,\ell - m}} \pi^{\shQ \,*}_{\lambda} \pi^{\shQ}_{\lambda\,*} = \bigoplus_{\lambda \in B_{d,\ell - m}} \pi_\lambda^* \, \pi_{\lambda\,*} \colon CH(\shZ_+) \to CH(\shZ_+)
	\end{align*}
is the same projector to the ``vector bundle part". We can also compute explicitly:
	\begin{align*}
		& \pi^{\shQ}_{\lambda\,*} (\blank) = \sum_{\mu \, \mid  \, \lambda \subseteq \mu \subseteq ((\ell-m)^d)} \Delta_{\mu/\lambda}(-E^\vee) \cdot \pi_{\mu\,*} (\blank), \\
		& =  \sum_{\mu} \Delta_{\mu/\lambda}(-E^\vee)  \cdot \sum_{\nu} \Delta_{\nu/\mu}(\sG^\vee) \cdot \pi_{*} (\sum_{\tau^c} \Delta_{\nu^c/\tau^c}(-E^\vee) \cdot \Delta_{\tau^c}(\shQ) \cap (\blank)) \\
	& =\sum_{\mu,\nu,\tau} \Delta_{\mu/\lambda}(-E^\vee) \cdot  \Delta_{\nu/\mu}(\sG^\vee)   \cdot \Delta_{\tau/\nu}(-E^\vee) \cdot  \pi_{*} (\Delta_{\tau^c}(\shQ) \cap (\blank)) \\
	& = \sum_{\tau \in B_{d,\ell - m} } \Delta_{\tau/\lambda}(-E^\vee - F^\vee)  \cdot \pi_{*} (\Delta_{\tau^c}(\shQ) \cap (\blank)).
		\end{align*}
Hence the explicit expression for $ \pi^{\shQ}_{\lambda\,*}$ is not as intrinsic as the one for $\pi_{\lambda\,*}$.
\end{remark}

\subsection{Lowest strata}
If $d \ge m$, over the lowest strata $Z$, $\shZ_{+}|_Z = : G_Z =  \Gr_d(E^\vee|_Z)$ is a Grassmannian $G_{d}(n)$-bundle. In this subsection we compute the contribution to Chow group from this strata. Notice if $m=0$ then $Z = \emptyset$, hence we need only consider $m \ge 1$.

\begin{lemma} \label{lem:lowest} For any $\lambda \in B_{d-m,\ell}$, $k \in \ZZ$, we define the following maps:
		\begin{align*} 
		&\Gamma_\lambda^*  := j_* (\Delta_{\lambda}(\shQ) \cap p^*(\blank))  \colon  & CH_{k-d\ell +|\lambda|}(Z)  \to CH_{k}(\shZ_+), \\
		&\Gamma_{\lambda \, *}  :=  \sum_{\mu \in B_{d-m,\ell}} \Delta_{\mu/\lambda}(-\sG^\vee) \cap p_* (\Delta_{\mu^c} (\shQ)\cap j^*(\blank)) \colon &  CH_{k}(\shZ_+) \to CH_{k-d\ell +|\lambda|}(Z),
		\end{align*}
where $-\sG^\vee  \equiv - E^\vee + F^\vee \in K_0(X)$ and $\mu^c = \ell - \mu$. Then the following holds:
	$$\Gamma_{\lambda \, *} \, \Gamma_{\mu}^* = (-1)^{\ell m} \cdot  \delta_{\lambda,\mu} \, \Id_{CH(Z)}, \qquad \forall \lambda,\mu \in B_{d-m, \ell}.$$
Hence we have a split-injective for the ``lowest strata part" of Chow:
	$$\bigoplus_{\lambda \in B_{d-m,\ell}} \Gamma_\lambda^*   \colon \bigoplus_{\lambda \in  B_{d-m,\ell}}  CH_{k-d\ell +|\lambda|}(Z) \hookrightarrow CH_k(\shZ_+). $$
Furthermore, the following holds:
	\begin{align}\label{eqn:gamma.vs.p}
	\Gamma_{\lambda \,*}  = \sum_{\mu \in B_{d-m,\ell}} \Delta_{\mu/\lambda}(F^\vee) \cdot p_{(\mu^t+m)^t  \,* }' \, j^*(\blank),
	\end{align}
where $p'_{\lambda \,*}$ is defined as in Lem. \ref{lem:Gr:projector}, i.e. with respect to the basis $\{\Delta_\lambda (\shQ)\}$.
\end{lemma}

\begin{proof} If we set $\nu = \lambda^t \in B_{\ell,d-m}$ and consider $\Delta_{\nu} (-\shQ^\vee) = \Delta_{\lambda} (\shQ)$, where $\shQ^\vee \simeq \shU(E)$ under $\Gr_d(E^\vee) \simeq \Gr_\ell(E)$. Since $j^* j_* (\blank) = c_{\rm top}(F^\vee \otimes \shQ^\vee) \cap (\blank)$, and $c_{\rm top}(F^\vee \otimes \shQ^\vee) = (-1)^{m\ell} c_{\rm top}(F \otimes \shQ) =  (-1)^{m\ell} \sum_{\nu \in B_{\ell,m}} \Delta_{\nu}(F) \cdot \Delta_{m-\nu}(-\shQ^\vee)$. Therefore the computation is exactly the same as Lem. \ref{lem:vb} up to the sign $(-1)^{\ell m}$, with the role of $\shU$ played by $\shQ^\vee$, and $E^\vee$ (resp. $F^\vee$) by $E$ (resp. $F$), and the basis $\{\Delta_\lambda(-\shU)\}$ by $\{ \Delta_{\nu} (-\shQ^\vee)\}$. 
\end{proof}

\begin{remark} \label{rmk:lem:lowest} Similar to top strata, we could also consider $\Gamma^{\,\shU \,*}_{\lambda}(\blank): = j_* (\Delta_\lambda(-\shU) \cap p^*(\blank)) $ for $\lambda \in B_{d-m, \ell}$. By change of basis $-\shU = -\shQ + E^\vee$,
	$$
		\Gamma^{\,\shU \,*}_{\lambda} (\blank)= \sum_{\mu \, \mid  \, \mu \subseteq \lambda} \Delta_{\lambda/\mu}(-E^\vee) \cdot \Gamma_\mu^* (\blank).
	$$
Therefore if we consider the following maps:
	\begin{align*}
		\Gamma^{\,\shU}_{\lambda\,*} (\blank) : = \sum_{\mu \, \mid  \, \lambda \subseteq \mu \subseteq ((\ell)^{d-m})} \Delta_{\mu/\lambda}(E^\vee) \cdot \Gamma_{\mu\,*} (\blank),
	\end{align*}
then $\Gamma^{\,\shU}_{\lambda\,*}  \, \Gamma^{\,\shU \,*}_{\mu}= (-1)^{\ell m} \, \delta_{\lambda,\mu} \, \Id_{CH(Z)}$ holds. Therefore $\bigoplus_{\lambda \in B_{d-m,\ell}} \Gamma^{\,\shU \,*}_{\lambda}$ also induces an embedding of Chow groups as  Lem. \ref{lem:lowest}, with the same image. Furthermore:
	\begin{align*}
	\bigoplus_{\lambda \in B_{d-m,\ell}} \Gamma^{\,\shU \,*}_{\lambda} \Gamma^{\,\shU}_{\lambda\,*} = \bigoplus_{\lambda \in B_{d-m,\ell}} \Gamma_\lambda^* \, \Gamma_{\lambda\,*} \colon CH(\shZ_+) \to CH(\shZ_+)
	\end{align*}
is the same projector to the ``lowest strata part". As before, we can explicitly compute:
	\begin{align*}
		\Gamma^{\,\shU}_{\lambda\,*} (\blank)  = \sum_{\mu \in B_{d-m,\ell} } \Delta_{\mu/\lambda}(E^\vee + F^\vee)  \cap p_{*} (\Delta_{\mu^c}(-\shU) \cap j^*(\blank)).
		\end{align*}
Hence the explicit expression for $\Gamma^{\,\shU}_{\lambda\,*}$ is slightly less intrinsic than that for $\Gamma_{\lambda\,*}$.
\end{remark}

\begin{lemma} For any $\lambda \in B_{d,\ell-m}$, $\mu \in B_{d-m,\ell}$, let $\pi_{\lambda}^*, \pi_{\lambda *}$, $ \pi^{\shQ \,*}_{\lambda}, \pi^{\shQ}_{\lambda\,*}$  be defined as in Lem. \ref{lem:vb}, Rmk. \ref{rmk:lem:vb}, and $\Gamma_{\mu}^*, \Gamma_{\mu *}$, $\Gamma^{\,\shU \,*}_{\mu}, \Gamma^{\,\shU}_{\mu\,*}$ be defined as in Lem. \ref{lem:lowest}, Rmk. \ref{rmk:lem:lowest}. Then:
	\begin{align*}
	\pi_{\lambda \,*} \,  \Gamma_{\mu}^*= \pi_{\lambda \,*} \, \Gamma^{\,\shU \,*}_{\mu} =  \pi^{\shQ}_{\lambda\,*}  \,  \Gamma_{\mu}^* =   \pi^{\shQ}_{\lambda\,*} \,  \Gamma^{\,\shU \,*}_{\mu}  =0, \qquad 
	\Gamma_{\mu \, *} \, \pi_{\lambda}^* =  \Gamma_{\mu \, *}  \,  \pi^{\shQ \,*}_{\lambda}= \Gamma^{\,\shU}_{\mu\,*}  \, \pi_{\lambda}^* =  \Gamma^{\,\shU}_{\mu\,*}  \, \pi^{\shQ \,*}_{\lambda}=   0.
	\end{align*}	
\end{lemma}
\begin{proof} For example, from (\ref{eqn:pi.vs.q}) of Lem. \ref{lem:vb}, we have
	\begin{align*}
		\pi_{\lambda \,*}\, \Gamma_{\mu \, *}^{\shU} =   \sum_{\nu \in B_{d,\ell - m}} \Delta_{\nu/\lambda}(-F^\vee)  \cap q_{\nu+m\, *} \iota_* j_* p_\mu^* =   \sum_{\nu \in B_{d,\ell - m}} \Delta_{\nu/\lambda}(-F^\vee)  \cap q_{\nu+m\, *} q_\mu^* i_* = 0,
	\end{align*}
where $\{ \nu + m \mid \nu \in B_{d,\ell -m} \} \cap B_{d-m,\ell} = \emptyset$ as long as $m \ge 1$. Similarly, from (\ref{eqn:gamma.vs.p}), 
	\begin{align*}
		\Gamma_{\mu \, *} \, \pi_{\lambda}^{\shQ\,*}  = \sum_{\nu \in B_{d-m,\ell}} \Delta_{\nu/\lambda}(F^\vee) \cdot p_{(\nu^t+m)^t  \,* } j^*\pi_{\lambda}^{\shQ\,*} = 
		 \sum_{\nu \in B_{d-m,\ell}} \Delta_{\nu/\lambda}(F^\vee) \cdot p_{(\nu^t+m)^t  \,* }' p_{\lambda}'^{*}  i^*= 0,
	\end{align*}
where $\{ (\nu^t + m)^t \mid \nu \in B_{d-m,\ell} \} \cap B_{d,\ell-m} = \emptyset$ as long as $m \ge 1$. All the other equalities follow from linear change of basis Rmk.  \ref{rmk:lem:vb} and \ref{rmk:lem:lowest}.
\end{proof}

\end{document}